\documentclass{amsart}

\usepackage{latexsym,amsfonts,amssymb,exscale,enumerate}
\usepackage{amsmath,amsthm,amscd}
\usepackage{hyperref}

\input xy
\usepackage[all]{xy}
\xyoption{line}
\xyoption{arrow}
\xyoption{color}
\SelectTips{cm}{}

\usepackage{tikz}
\usetikzlibrary{decorations.markings}
\usetikzlibrary{decorations.pathreplacing}

\tikzset{->-/.style={decoration={
  markings,
  mark=at position #1 with {\arrow{>}}},postaction={decorate}}}

\tikzset{middlearrow/.style={
        decoration={markings,
            mark= at position 0.5 with {\arrow{#1}} ,
        },
        postaction={decorate}
    }
}

\usepackage{graphicx}
\usepackage{color}

\newcommand{\onenn}[1]{{\mathbf 1}_{#1}}

\newcommand{\arxiv}[1]{\href{http://arxiv.org/abs/#1}{\tt arXiv:\nolinkurl{#1}}}

\theoremstyle{plain}

\theoremstyle{definition}

\theoremstyle{remark}

\theoremstyle{definition}
\newtheorem{thm}{Theorem}[section]
\newtheorem{cor}[thm]{Corollary}

\newtheorem{lem}[thm]{Lemma}
\newtheorem{rem}[thm]{Remark}
\newtheorem{prop}[thm]{Proposition}
\newtheorem{defn}[thm]{Definition}

\newcommand{\Ucat}{\cal{U}}

\newcommand{\UcatD}{\dot{\cal{U}}}

\newcommand{\UA}{{_{\cal{A}}\dot{{\bf U}}}}

\newcommand{\maps}{\colon}

\newcommand{\xsum}[2]{
  \xy
  (0,.4)*{\sum};
  (0,3.7)*{\scs #2};
  (0,-2.9)*{\scs #1};
  \endxy
}
\newcommand{\xprod}[2]{
  \xy
  (0,.4)*{\prod};
  (0,3.7)*{\scs #2};
  (0,-2.9)*{\scs #1};
  \endxy
}

\newcommand{\refequal}[1]{\xy {\ar@{=}^{#1}
(-1,0)*{};(1,0)*{}};
\endxy}

\newcommand{\Hom}{{\rm Hom}}

\renewcommand{\to}{\rightarrow}
\newcommand{\excise}[1]{}

\def\Id{\mathrm{Id}}

\def\mf{\mathfrak}
\def\shuffle{\,\raise 1pt\hbox{$\scriptscriptstyle\cup{\mskip
               -4mu}\cup$}\,}

\newcommand{\ii}{ \textbf{\textit{i}}}

\numberwithin{equation}{section}

\def\emph#1{{\sl #1\/}}

\let\tilde=\widetilde

\let\theta=\vartheta
\let\epsilon=\varepsilon

\usepackage{bbm}

\def\Z{{\mathbbm Z}}
\def\Q{{\mathbbm Q}}

\def\cal#1{\mathcal{#1}}%
\def\1{\mathbbm{1}}%
\def\nn{\notag}

\def\la{\langle}
\def\ra{\rangle}

\newcommand{\ccbub}[1]{
\xybox{%
 (-6,0)*{};
  (6,0)*{};
  (-4,0)*{}="t1";
  (4,0)*{}="t2";
  "t2";"t1" **\crv{(4,6) & (-4,6)};
   ?(1)*\dir{>};
  "t2";"t1" **\crv{(4,-6) & (-4,-6)};
   ?(.3)*\dir{}+(0,0)*{\bullet}+(0,-3)*{\scs {#1}};
}}

\newcommand{\cbub}[1]{
\xybox{%
 (-6,0)*{};
  (6,0)*{};
  (-4,0)*{}="t1";
  (4,0)*{}="t2";
  "t2";"t1" **\crv{(4,6) & (-4,6)};
    ?(.95)*\dir{<};
  "t2";"t1" **\crv{(4,-6) & (-4,-6)};
   ?(.3)*\dir{}+(0,0)*{\bullet}+(0,-3)*{\scs {#1}};
}}
\newcommand{\bbe}[1]{\xybox{%
  (-2,0)*{};
  (2,0)*{};
  (0,0);(0,-18) **\dir{-}; ?(.5)*\dir{<}+(2.3,0)*{\scriptstyle{#1}};
}}

\newcommand{\bbpef}{\xybox{%
  (-6,0)*{};
  (6,0)*{};
  (-4,0)*{}="t1";
  (4,0)*{}="t2";
  "t1";"t2" **\crv{(-4,-6) & (4,-6)}; ?(.15)*\dir{>} ?(.9)*\dir{>};
}}
\newcommand{\bbpfe}{\xybox{%
  (-6,0)*{};
  (6,0)*{};
  (-4,0)*{}="t1";
  (4,0)*{}="t2";
  "t2";"t1" **\crv{(4,-6) & (-4,-6)}; ?(.15)*\dir{>} ?(.9)*\dir{>};
}}

\newcommand{\bbcfe}[1]{\xybox{%
  (-6,0)*{};
  (6,0)*{};
  (-4,0)*{}="t1";
  (4,0)*{}="t2";
  "t1";"t2" **\crv{(-4,6) & (4,6)}; ?(.15)*\dir{>} ?(.9)*\dir{>}
  ?(.5)*\dir{}+(0,2)*{\scriptstyle{#1}};
}}
\newcommand{\bbcef}[1]{\xybox{%
  (-6,0)*{};
  (6,0)*{};
  (-4,0)*{}="t1";
  (4,0)*{}="t2";
  "t2";"t1" **\crv{(4,6) & (-4,6)}; ?(.15)*\dir{>}
  ?(.9)*\dir{>} ?(.5)*\dir{}+(0,2)*{\scriptstyle{#1}};
}}

\newcommand{\lowrru}[1]{\xybox{%
  (-8,0)*{};
  (8,0)*{};
  (-6,-18)*{};(6,-9)*{} **\crv{(-6,-13) & (6,-15)} ?(1)*\dir{>};
  (6,-9)*{};(6,0)*{}  **\dir{-} ?(.3)*\dir{ }+(2,0)*{\scs {\bf j}};
}}

\newcommand{\lowllu}[1]{\xybox{%
  (-8,0)*{};
  (8,0)*{};
  (6,-18)*{};(-6,-9)*{} **\crv{(6,-13) & (-6,-15)} ?(1)*\dir{>};
  (-6,-9)*{};(-6,0)*{}  **\dir{-} ?(.3)*\dir{ }+(-2,0)*{\scs {\bf j}};
}}

\newcommand{\bbdl}[1]{\xybox{%
  (2,0);(0,-8) **\crv{(2,-2)&(0,-6)}; ?(.5)*\dir{>}
}}
\newcommand{\bbdlu}[1]{\xybox{%
  (2,0);(0,-8) **\crv{(2,-2)&(0,-6)}; ?(.5)*\dir{<}
}}
\newcommand{\bbdr}[1]{\xybox{%
  (-2,0);(0,-8) **\crv{(-2,-2)&(0,-6)}; ?(.5)*\dir{>}
}}
\newcommand{\bbdru}[1]{\xybox{%
  (-2,0);(0,-8) **\crv{(-2,-2)&(0,-6)}; ?(.5)*\dir{<}
}}

\DeclareGraphicsExtensions{.pdf,.eps}
\usepackage{psfrag}
\usepackage{bbm}

\def\cal#1{\mathcal{#1}}

\def \Z {\mathbbm{Z}}

\def \Q {\mathbbm{Q}}

\def \Tr{\operatorname{Tr}}

\def \Id {{\rm Id}}

\newcommand\nc{\newcommand}
\nc\rnc{\renewcommand}
\nc\Kar{\operatorname{Kar}}
\nc\End{\operatorname{End}}
\nc\modQ {{\mathbb Q}}
\nc\modZ {{\mathbb Z}}
\nc\simeqto{\overset{\simeq}{\longrightarrow }}
\nc\modC {{\mathcal C}}
\nc\modD {{\mathcal D}}
\nc\K{\mathcal {K}}
\nc\CC{\mathbf{C}}

\newcommand{\scs}{\scriptstyle}

\nc\calU{\mathcal{U}}
\nc\cU{\calU}
\nc\col{\colon\thinspace}
\nc\calA{\mathcal{A}}
\nc\Ab{\mathbf{Ab}}
\nc\Ko{K_0}
\nc\TrhorCC{\Tr^{\mathrm{hor}}(\CC)}
\nc\AdCat{\mathbf{AdCat}}
\nc\TrCC{\Tr(\CC)}
\nc\Udot{\dot{\mathcal{U}}}
\nc\diag{\mathrm{d}}
\nc\modU {\mathcal{U}}
\nc\bfU{\mathbf{U}}
\nc\dU{\dot{\mathbf U}}
\nc\dUZ{{_\modZ\dot{\mathbf U}}}
\nc\UZ{{_\modZ \mathbf U} }
\nc\fsl{\mathfrak{sl}}
\nc\Uaa{{\bf U} (\mathfrak{sl}_2\otimes \Q[t,t^{-1}])}
\nc\UZslt{{_\modZ\mathbf{U}} (\mathfrak{sl}_2\otimes \Q[t])}
\nc\UdZslt{{_\modZ\dot{\mathbf{U}}} (\mathfrak{sl}_2\otimes \Q[t])}
\nc\LL{L^+\fsl_2}
\nc\UL{\mathbf U(\LL)}
\nc\UZL{\UZ(\LL)}
\nc\dUZL{\dUZ(\LL)}
\nc\dUL{\dU(\LL)}
\nc{\im}{\rm im}
\nc\Kom{\rm Kom}
\nc\GL{\rm{GL}}
\nc\g{\mathfrak{g}}

\nc\tG{\tilde{G}} \nc\tE{\tilde{E}}
\nc\Vect{\rm Vect}
\nc{\Gras}{{\rm {Gr}}}
\nc\FMod{\rm FMod}
\nc\yto[1]{\underset{#1}{\to}}
\nc\Ear{\yto{E}}

\newcommand\sE{{\cal{E}}}
\newcommand\sF{{\cal{F}}}

\def\l{\lambda}
\newcommand{\iccbub}[2]{
\xybox{%
 (-6,0)*{};
  (6,0)*{};
  (-4,0)*{}="t1";
  (4,0)*{}="t2";
  "t2";"t1" **\crv{(4,6) & (-4,6)}; ?(.7)*\dir{}+(-2,0)*{\scs #2}
  ?(.05)*\dir{>} ?(1)*\dir{>};
  "t2";"t1" **\crv{(4,-6) & (-4,-6)};
   ?(.3)*\dir{}+(0,0)*{\bullet}+(0,-3)*{\scs {#1}};
}}
\newcommand{\icbub}[2]{
\xybox{%
 (-6,0)*{};
  (6,0)*{};
  (-4,0)*{}="t1";
  (4,0)*{}="t2";
  "t2";"t1" **\crv{(4,6) & (-4,6)};?(.7)*\dir{}+(-2,0)*{\scs #2};
   ?(0)*\dir{<} ?(.95)*\dir{<};
  "t2";"t1" **\crv{(4,-6) & (-4,-6)};
   ?(.3)*\dir{}+(0,0)*{\bullet}+(0,-3)*{\scs {#1}};
}}

\newcommand{\onel}{{\mathbf 1}_{\lambda}}

\allowdisplaybreaks

\begin{document}


\date{\today}

\title{Cyclicity for categorified quantum groups}

\author{Anna Beliakova}
\address{Universit\"at Z\"urich, Winterthurerstr. 190
CH-8057 Z\"urich, Switzerland}
\email{anna@math.uzh.ch}

\author{Kazuo Habiro}
\address{Research Institute for Mathematical Sciences, Kyoto University, Kyoto, 606-8502, Japan}
\email{habiro@kurims.kyoto-u.ac.jp}

\author{Aaron D.~Lauda}
\address{University of Southern California, Los Angeles, CA 90089, USA}
\email{lauda@usc.edu}

\author{Ben Webster}
\address{University of Virginia,
  Charlottesville, VA 22903, USA}
\email{bwebster@virginia.edu}

\begin{abstract}
We equip the categorified quantum group attached KLR algebra and an
arbitrary choice of scalars with duality functor which is cyclic, that
is, such that $f=f^{**}$
for all 2-morphisms $f$.  This is accomplished via a modified
diagrammatic formalism.
\end{abstract}

\maketitle


Consider the 2-category $\Ucat_Q(\mf{g})$ (as defined in
\cite{Brundan2,CLau,Rou2}) which categorifies a
Kac-Moody Lie algebra $\mf{g}$.  In this 2-category, every 1-morphism
possesses a left dual and a right dual.  Standard arguments show that
the left and right duals (which are isomorphic) are well-defined
up to unique isomorphism; however they are not well-defined as
1-morphisms.  That is, the operations of left and right dual are
anafunctors, not functors.  In particular, every 1-morphism is
isomorphic to its double-dual, but there is no fixed
isomorphism intrinsic to the 2-category structure.

However, we can
choose a right duality functor, which allows us to fix a left and
right dual.  In fact, this duality will be strict: any 1-morphism will be
equal to its double dual.   One such functor is supplied by the
functor $\tilde\tau$ defined in \cite[3.46]{KL3}, that is by rightward
rotation by $180^\circ$.  However, in the graphical calculus defined
in \cite{CLau}, this functor lacks one of the basic properties we
expect from a duality: while a 1-morphism is equal to its double dual,
for a 2-morphism $f\colon u\to v$, we will not necessarily have that $f$
and $f^{**}$ are equal.

If the equality $f=f^{**}$ does hold, we call the resulting duality
{\bf cyclic}.  In this case, the functor of double dual
is isomorphic to the identity; thus, the identity map defines a
{\bf (strictly)  pivotal} structure on this 2-category.  This property is also
equivalent to the induced biadjunction on the objects $(u,u^*)$ being
cyclic.  In diagrammatic terms, this means that a morphism is
unchanged by a full $360^\circ$ rotation.  It follows that
the string diagram calculus which uses a cyclic duality functor
enjoys a topological invariance that greatly simplifies computations.
Furthermore, a pivotal structure is necessary for defining convolution in the
Hochschild cohomology of a representation of this 2-category; in
particular, it plays a key role in the authors' previous work on
traces of categorified quantum groups, and related work of Shan,
Varagnolo and Vasserot~\cite{BHLW,SVV,Web6}.

Our work is motivated by work of Brundan~\cite{Brundan2}, which shows
that the 2-category introduced in \cite{CLau} can be defined as in
\cite{Rou2}, where the functor $\sF_i$ is the right dual of $\sE_i$ (by
definition), but there is no {\it a priori} connection between  $\sF_i$ and the left dual of $\sE_i$.  In order to define a duality functor, we must thus
choose an isomorphism of $\sE_i$ to the right dual of $\sF_i$ (that is, of
$\sF_i$ with the left dual of $\sE_i$).

In \cite[1.2]{Brundan2}, one
such isomorphism is defined, which matches the choice implicit in
\cite{CLau}.  Unfortunately, as \cite[(2.5)]{CLau} shows, this duality
is usually
not cyclic; in particular, it is not for a generic choice of parameters, or for
the choice which is most important in geometric and representation
theoretic applications, such as \cite{BK1,VV}.

In this paper, we introduce a 2-category
$\Ucat_Q^{cyc}(\mf{g})$ for an arbitrary choice of parameters $Q$ with
canonical cyclic duality, which is equivalent to
$\Ucat_Q(\mf{g})$ as a 2-category; of course, this equivalence is not
compatible with the duality functor.  That is, we give a diagrammatic framework for this
2-category where arbitrary isotopies leave 2-morphisms unchanged.
This differs from the similar calculus given in
\cite{CLau} in that
the ``degree zero bubbles"
are no longer identity 2-morphisms, but rather a multiple of the
identity.  By carefully choosing these parameters we are able to
modify the biadjunction making the 2-category $\Ucat_Q(\mf{g})$
pivotal.  Aside from this modification of bubble parameters, the
resulting 2-category $\Ucat_Q^{cyc}(\mf{g})$ has surprisingly simple
relations including a simplified version of the mixed
$\cal{E}_i\cal{F}_j$ relations (see \eqref{mixed_rel-cyc}) and nice
rotation properties for sideways crossings (see
\eqref{eq_crossl-gen-cyc} and \eqref{eq_crossr-gen-cyc}).
Furthermore, in the 2-category $\Ucat_Q(\mf{g})$, if the KLR algebra
$R_Q$ associated to a choice of scalars $Q$ governs the endomorphism
algebras of the upward oriented strings (endomorphisms in
$\Ucat_Q^{+}(\mf{g})$), then a different KLR algebra $R_{Q'}$ for a
related choice of scalars $Q'$ governs the downward oriented strings
(endomorphisms in $\Ucat_Q^{-}(\mf{g})$).  In our new pivotal
2-category $\Ucat_Q^{cyc}(\mf{g})$ the same KLR algebra governs the
upward and downward oriented graphical calculus.  This 2-category has some similar features as the 2-category defined for $\mf{gl}_n$ in \cite{MSV}.

In Section~\ref{sec:iso} we provide an explicit isomorphism between $\Ucat_Q^{cyc}(\mf{g})$ and the 2-category $\Ucat_Q(\mf{g})$ from \cite{CLau}.
This isomorphism is a straightforward scalar multiplication, so no
deep new techniques were needed in the construction of $\Ucat_Q^{cyc}(\mf{g})$.
However, it was only in the light of Brundan's rigidity theorem that this modification became apparent to the authors.

Our aim in this note is to provide the most convenient definition of categorified quantum groups to aid researchers in the field of higher representation theory.  To that end, we provide some useful relations in section~\ref{sec:relations} that can be derived in $\Ucat_Q^{cyc}(\mf{g})$.  It is clear that the 2-category $\Ucat_Q^{cyc}(\mf{g})$ introduced here is the most natural version of the categorified quantum group for a general choice of scalars $Q$.  Indeed, the 2-isomorphism defined in section~\ref{sec:iso} removes some of the inconvenient signs that have appeared in the study of various 2-representations~\cite{LQR,QR,MY}.

\medskip
{\bf Acknowledgments.}
The authors are grateful to Mike Abram and Laffite Lamberto-Egan for helpful comments on a preliminary version of this paper.   A.B. was  supported by Swiss National Science Foundation under Grant PDFMP2-141752/1.
K.H. was supported by JSPS KAKENHI Grant Numbers 24540077, 15K04873. A.D.L. was partially supported by NSF grant DMS-1255334, the John Templeton Foundation, and the SwissMAP NCCR grant of  Swiss National Science Foundation.  B.W was supported by the NSF under Grant DMS-1151473.

%
\section{The cyclic categorified quantum group}
\label{categorified-quantum}
%

%
\subsubsection{The Cartan datum and choice of scalars $Q$}\label{sec:datum}
%

We fix a Cartan datum consisting of
\begin{itemize}
\item a free $\Z$-module $X$ (the weight lattice),
\item for $i \in I$ ($I$ is an indexing set) there are elements $\alpha_i \in X$ (simple roots) and $\Lambda_i \in X$ (fundamental weights),
\item for $i \in I$ an element $h_i \in X^\vee = \Hom_{\Z}(X,\Z)$ (simple coroots),
\item a bilinear form $(\cdot,\cdot )$ on $X$.
\end{itemize}
Write $\langle \cdot, \cdot \rangle \maps X^{\vee} \times X
\to \Z$ for the canonical pairing. These data should satisfy:
\begin{itemize}
\item $(\alpha_i, \alpha_i) \in 2\Z_{>0}$ for any $i\in I$,
\item $\la i,\lambda\ra :=\langle h_i, \lambda \rangle = 2 \frac{(\alpha_i,\lambda)}{(\alpha_i,\alpha_i)}$
  for $i \in I$ and $\lambda \in X$,
\item $(\alpha_i,\alpha_j) \leq 0$  for $i,j\in I$ with $i \neq j$,
\item $\langle h_j, \Lambda_i \rangle =\delta_{ij}$ for all $i,j \in I$.
\end{itemize}
Hence $C_{i,j}=\{\langle h_i, \alpha_j \rangle\}_{i,j\in I}$ is a symmetrizable generalized Cartan matrix. In what follows we write
\begin{gather}
 d_{ij} =-\langle i, \alpha_j \rangle,\\
 d_i=\frac{(\alpha_i,\alpha_i)}{2}.
\end{gather}
for $i,j\in I$.

\medskip

Let $\Bbbk$ be a commutative ring with unit, and let $\Bbbk^{\times}$
denote the group of units in $\Bbbk$.

\begin{defn}
Associated to a Cartan datum we also fix a {\em choice of scalars $Q$} consisting of
\begin{itemize}
  \item $t_{ij}$ for all $i,j \in I$ ,
  \item $s_{ij}^{pq}\in \Bbbk$ for $i \neq j$,  $0 \leq p < d_{ij}$, and $0 \leq q < d_{ji}$,
\end{itemize}
such that
\begin{itemize}
\item $t_{ii}=1$ for all $i \in I$ and $t_{ij} \in \Bbbk^{\times}$ for $i\neq j$,
 \item $s_{ij}^{pq}=s_{ji}^{qp}$,
 \item $t_{ij}=t_{ji}$ when $d_{ij}=0$.
\end{itemize}
We set $s_{ij}^{pq}=0$ when $p,q<0$ or $d_{ij} \geq p$ or $d_{ji} \geq q$.
\end{defn}

\begin{defn}
A choice of {\em bubble parameters $C$} compatible with the choice
of scalars $Q$ is a set consisting of
\begin{itemize}
\item $c_{i,\l} \in \Bbbk^{\times}$  for $i\in I$ and $\lambda \in X$.
\end{itemize}
such that
\begin{equation}
  c_{i,\lambda+\alpha_j}/c_{i,\lambda}=t_{ij}
\end{equation}
for $i,j\in I$, $\lambda\in X$.
\end{defn}

Of course, we can choose such a set of scalars for any $t_{ij}$ by
fixing an arbitrary choice of $c_{i,\lambda}$ for a fixed $\lambda$
in each coset of the root lattice in the weight lattice, and then
extending to the rest of the coset using the compatibility condition.

For any choice of bubble parameters compatible with the choice of scalars $Q$ the values along an $\mf{sl}_2$-string remain constant since
\[
 c_{i,\l+\alpha_i} = t_{ii} c_{i,\l} = c_{i,\l},
\]
so that for all $n \in \Z$ we have $c_{i,\l} = c_{i,\l+n\alpha_i}$.

%
\subsection{Definition of the 2-category  $\Ucat_Q^{cyc}(\mf{g})$}
%

By a graded linear 2-category we mean a category enriched in graded linear categories, so that the hom spaces form graded linear categories, and the composition map is grading preserving.  Given a fixed choice of scalars $Q$ and a compatible choice of bubble parameters $C$ we can define the following 2-category.

\begin{defn} \label{defU_cat-cyc}
The 2-category $\Ucat_Q^{cyc}(\mf{g})$ is the graded linear 2-category consisting of:
\begin{itemize}
\item objects $\lambda$ for $\lambda \in X$.
\item 1-morphisms are formal direct sums of (shifts of) compositions of
$$\onel, \quad \onenn{\l+\alpha_i} \sE_i = \onenn{\l+\alpha_i} \sE_i\onel = \sE_i \onel, \quad \text{ and }\quad \onenn{\lambda-\alpha_i} \sF_i = \onenn{\lambda-\alpha_i} \sF_i\onel = \sF_i\onel$$
for $i \in I$ and $\l \in X$.
In particular, any morphism can be written as a finite formal sum of symbols $\sE_{\ii}\onel\la t\ra$ where $\ii = (\pm i_1, \dots, \pm i_m)$ is a signed sequence of simple roots, $t$ is a grading shift,  $\sE_{+i}\onel := \sE_i\onel$ and $\sE_{-i}\onel:=\sF_i\onel$, and $\sE_{\ii}\onel\la t\ra := \sE_{\pm i_1} \dots \sE_{\pm i_m}\onel \la t\ra$.

\item 2-morphisms are $\Bbbk$-modules spanned by compositions of (decorated) tangle-like diagrams illustrated below.

\begin{align}
  \xy 0;/r.17pc/:
 (0,7);(0,-7); **\dir{-} ?(.75)*\dir{>};
 (0,0)*{\bullet};
 (7,3)*{ \scs \lambda};
 (-9,3)*{\scs  \lambda+\alpha_i};
 (-2.5,-6)*{\scs i};
 (-10,0)*{};(10,0)*{};
 \endxy &\maps \cal{E}_i\onel \to \cal{E}_i\onel\la (\alpha_i,\alpha_i) \ra  & \quad
 &
    \xy 0;/r.17pc/:
 (0,7);(0,-7); **\dir{-} ?(.75)*\dir{<};
 (0,0)*{\bullet};
 (7,3)*{ \scs \lambda};
 (-9,3)*{\scs  \lambda-\alpha_i};
 (-2.5,-6)*{\scs i};
 (-10,0)*{};(10,0)*{};
 \endxy\maps \cal{F}_i\onel \to \cal{F}_i\onel\la (\alpha_i,\alpha_i) \ra  \nn \\
   & & & \nn \\
   \xy 0;/r.17pc/:
  (0,0)*{\xybox{
    (-4,-4)*{};(4,4)*{} **\crv{(-4,-1) & (4,1)}?(1)*\dir{>} ;
    (4,-4)*{};(-4,4)*{} **\crv{(4,-1) & (-4,1)}?(1)*\dir{>};
    (-5.5,-3)*{\scs i};
     (5.5,-3)*{\scs j};
     (9,1)*{\scs  \lambda};
     (-10,0)*{};(10,0)*{};
     }};
  \endxy \;\;&\maps \cal{E}_i\cal{E}_j\onel  \to \cal{E}_j\cal{E}_i\onel\la - (\alpha_i,\alpha_j) \ra  &
  &
   \xy 0;/r.17pc/:
  (0,0)*{\xybox{
    (-4,4)*{};(4,-4)*{} **\crv{(-4,1) & (4,-1)}?(1)*\dir{>} ;
    (4,4)*{};(-4,-4)*{} **\crv{(4,1) & (-4,-1)}?(1)*\dir{>};
    (-6.5,-3)*{\scs i};
     (6.5,-3)*{\scs j};
     (9,1)*{\scs  \lambda};
     (-10,0)*{};(10,0)*{};
     }};
  \endxy\;\; \maps \cal{F}_i\cal{F}_j\onel  \to \cal{F}_j\cal{F}_i\onel\la - (\alpha_i,\alpha_j) \ra  \nn \\
     & & & \nn \\
   \xy 0;/r.17pc/:
  (0,0)*{\xybox{
    (-4,-4)*{};(4,4)*{} **\crv{(-4,-1) & (4,1)}?(0)*\dir{<} ;
    (4,-4)*{};(-4,4)*{} **\crv{(4,-1) & (-4,1)}?(1)*\dir{>};
    (-5.5,-3)*{\scs i};
     (5.5,-3)*{\scs j};
     (9,1)*{\scs  \lambda};
     (-10,0)*{};(10,0)*{};
     }};
  \endxy \;\;&\maps \cal{F}_i\cal{E}_j\onel  \to \cal{E}_j\cal{F}_i\onel   &
  &
   \xy 0;/r.17pc/:
  (0,0)*{\xybox{
    (-4,4)*{};(4,-4)*{} **\crv{(-4,1) & (4,-1)}?(1)*\dir{>} ;
    (4,4)*{};(-4,-4)*{} **\crv{(4,1) & (-4,-1)}?(0)*\dir{<};
    (-6.5,-3)*{\scs i};
     (6.5,-3)*{\scs j};
     (9,1)*{\scs  \lambda};
     (-10,0)*{};(10,0)*{};
     }};
  \endxy\;\; \maps \cal{E}_i\cal{F}_j\onel  \to \cal{F}_j\cal{E}_i\onel   \nn \\
  & & & \nn \\
     \xy 0;/r.17pc/:
    (0,-3)*{\bbpef{i}};
    (8,-5)*{\scs  \lambda};
    (-10,0)*{};(10,0)*{};
    \endxy &\maps \onel  \to \cal{F}_i\cal{E}_i\onel\la d_i + (\l, \alpha_i) \ra   &
    &
   \xy 0;/r.17pc/:
    (0,-3)*{\bbpfe{i}};
    (8,-5)*{\scs \lambda};
    (-10,0)*{};(10,0)*{};
    \endxy \maps \onel  \to\cal{E}_i\cal{F}_i\onel\la d_i - (\l, \alpha_i) \ra  \nn \\
      & & & \nn \\
  \xy 0;/r.17pc/:
    (0,0)*{\bbcef{i}};
    (8,4)*{\scs  \lambda};
    (-10,0)*{};(10,0)*{};
    \endxy & \maps \cal{F}_i\cal{E}_i\onel \to\onel\la d_i + (\l, \alpha_i) \ra  &
    &
 \xy 0;/r.17pc/:
    (0,0)*{\bbcfe{i}};
    (8,4)*{\scs  \lambda};
    (-10,0)*{};(10,0)*{};
    \endxy \maps\cal{E}_i\cal{F}_i\onel  \to\onel\la d_i - (\l, \alpha_i) \ra \nn
\end{align}
\end{itemize}
\end{defn}
In this $2$-category (and those throughout the paper) we
read diagrams from right to left and bottom to top.  The identity 2-morphism of the 1-morphism
$\cal{E}_i \onel$ is
represented by an upward oriented line labeled by $i$ and the identity 2-morphism of $\cal{F}_i \onel$ is
represented by a downward such line.

The 2-morphisms satisfy the following relations:
\begin{enumerate}
\item \label{item_cycbiadjoint-cyc} The 1-morphisms $\cal{E}_i \onel$ and $\cal{F}_i \onel$ are biadjoint (up to a specified degree shift). These conditions are expressed diagrammatically as
    \begin{equation} \label{eq_biadjoint1-cyc}
 \xy   0;/r.17pc/:
    (-8,0)*{}="1";
    (0,0)*{}="2";
    (8,0)*{}="3";
    (-8,-10);"1" **\dir{-};
    "1";"2" **\crv{(-8,8) & (0,8)} ?(0)*\dir{>} ?(1)*\dir{>};
    "2";"3" **\crv{(0,-8) & (8,-8)}?(1)*\dir{>};
    "3"; (8,10) **\dir{-};
    (12,-9)*{\lambda};
    (-6,9)*{\lambda+\alpha_i};
    \endxy
    \; =
    \;
\xy   0;/r.17pc/:
    (-8,0)*{}="1";
    (0,0)*{}="2";
    (8,0)*{}="3";
    (0,-10);(0,10)**\dir{-} ?(.5)*\dir{>};
    (5,8)*{\lambda};
    (-9,8)*{\lambda+\alpha_i};
    \endxy
\qquad \quad \xy  0;/r.17pc/:
    (8,0)*{}="1";
    (0,0)*{}="2";
    (-8,0)*{}="3";
    (8,-10);"1" **\dir{-};
    "1";"2" **\crv{(8,8) & (0,8)} ?(0)*\dir{<} ?(1)*\dir{<};
    "2";"3" **\crv{(0,-8) & (-8,-8)}?(1)*\dir{<};
    "3"; (-8,10) **\dir{-};
    (12,9)*{\lambda+\alpha_i};
    (-6,-9)*{\lambda};
    \endxy
    \; =
    \;
\xy  0;/r.17pc/:
    (8,0)*{}="1";
    (0,0)*{}="2";
    (-8,0)*{}="3";
    (0,-10);(0,10)**\dir{-} ?(.5)*\dir{<};
    (9,-8)*{\lambda+\alpha_i};
    (-6,-8)*{\lambda};
    \endxy
\end{equation}

\begin{equation}\label{eq_biadjoint2-cyc}
 \xy   0;/r.17pc/:
    (8,0)*{}="1";
    (0,0)*{}="2";
    (-8,0)*{}="3";
    (8,-10);"1" **\dir{-};
    "1";"2" **\crv{(8,8) & (0,8)} ?(0)*\dir{>} ?(1)*\dir{>};
    "2";"3" **\crv{(0,-8) & (-8,-8)}?(1)*\dir{>};
    "3"; (-8,10) **\dir{-};
    (12,9)*{\lambda};
    (-5,-9)*{\lambda+\alpha_i};
    \endxy
    \; =
    \;
    \xy 0;/r.17pc/:
    (8,0)*{}="1";
    (0,0)*{}="2";
    (-8,0)*{}="3";
    (0,-10);(0,10)**\dir{-} ?(.5)*\dir{>};
    (5,-8)*{\lambda};
    (-9,-8)*{\lambda+\alpha_i};
    \endxy
\qquad \quad \xy   0;/r.17pc/:
    (-8,0)*{}="1";
    (0,0)*{}="2";
    (8,0)*{}="3";
    (-8,-10);"1" **\dir{-};
    "1";"2" **\crv{(-8,8) & (0,8)} ?(0)*\dir{<} ?(1)*\dir{<};
    "2";"3" **\crv{(0,-8) & (8,-8)}?(1)*\dir{<};
    "3"; (8,10) **\dir{-};
    (12,-9)*{\lambda+\alpha_i};
    (-6,9)*{\lambda};
    \endxy
    \; =
    \;
\xy   0;/r.17pc/:
    (-8,0)*{}="1";
    (0,0)*{}="2";
    (8,0)*{}="3";
    (0,-10);(0,10)**\dir{-} ?(.5)*\dir{<};
   (9,8)*{\lambda+\alpha_i};
    (-6,8)*{\lambda};
    \endxy
\end{equation}

  \item The 2-morphisms are cyclic with respect to this biadjoint structure.
\begin{equation} \label{eq_cyclic_dot-cyc}
 \xy 0;/r.17pc/:
    (-8,5)*{}="1";
    (0,5)*{}="2";
    (0,-5)*{}="2'";
    (8,-5)*{}="3";
    (-8,-10);"1" **\dir{-};
    "2";"2'" **\dir{-} ?(.5)*\dir{<};
    "1";"2" **\crv{(-8,12) & (0,12)} ?(0)*\dir{<};
    "2'";"3" **\crv{(0,-12) & (8,-12)}?(1)*\dir{<};
    "3"; (8,10) **\dir{-};
    (17,-9)*{\lambda+\alpha_i};
    (-12,9)*{\lambda};
    (0,4)*{\bullet};
    (10,8)*{\scs };
    (-10,-8)*{\scs };
    \endxy
    \quad = \quad
      \xy 0;/r.17pc/:
 (0,10);(0,-10); **\dir{-} ?(.75)*\dir{<}+(2.3,0)*{\scriptstyle{}}
 ?(.1)*\dir{ }+(2,0)*{\scs };
 (0,0)*{\bullet};
 (-6,5)*{\lambda};
 (10,5)*{\lambda+\alpha_i};
 (-10,0)*{};(10,0)*{};(-2,-8)*{\scs };
 \endxy
    \quad = \quad
   \xy 0;/r.17pc/:
    (8,5)*{}="1";
    (0,5)*{}="2";
    (0,-5)*{}="2'";
    (-8,-5)*{}="3";
    (8,-10);"1" **\dir{-};
    "2";"2'" **\dir{-} ?(.5)*\dir{<};
    "1";"2" **\crv{(8,12) & (0,12)} ?(0)*\dir{<};
    "2'";"3" **\crv{(0,-12) & (-8,-12)}?(1)*\dir{<};
    "3"; (-8,10) **\dir{-};
    (17,9)*{\lambda+\alpha_i};
    (-12,-9)*{\lambda};
    (0,4)*{\bullet};
    (-10,8)*{\scs };
    (10,-8)*{\scs };
    \endxy
\end{equation}
The cyclic relations for crossings are given by
\begin{equation} \label{eq_cyclic}
   \xy 0;/r.17pc/:
  (0,0)*{\xybox{
    (-4,4)*{};(4,-4)*{} **\crv{(-4,1) & (4,-1)}?(1)*\dir{>} ;
    (4,4)*{};(-4,-4)*{} **\crv{(4,1) & (-4,-1)}?(1)*\dir{>};
    (-6.5,-3)*{\scs i};
     (6.5,-3)*{\scs j};
     (9,1)*{\scs  \lambda};
     (-10,0)*{};(10,0)*{};
     }};
  \endxy \quad = \quad
  \xy 0;/r.17pc/:
  (0,0)*{\xybox{
    (4,-4)*{};(-4,4)*{} **\crv{(4,-1) & (-4,1)}?(1)*\dir{>};
    (-4,-4)*{};(4,4)*{} **\crv{(-4,-1) & (4,1)};
     (-4,4)*{};(18,4)*{} **\crv{(-4,16) & (18,16)} ?(1)*\dir{>};
     (4,-4)*{};(-18,-4)*{} **\crv{(4,-16) & (-18,-16)} ?(1)*\dir{<}?(0)*\dir{<};
     (-18,-4);(-18,12) **\dir{-};(-12,-4);(-12,12) **\dir{-};
     (18,4);(18,-12) **\dir{-};(12,4);(12,-12) **\dir{-};
     (22,1)*{ \lambda};
     (-10,0)*{};(10,0)*{};
     (-4,-4)*{};(-12,-4)*{} **\crv{(-4,-10) & (-12,-10)}?(1)*\dir{<}?(0)*\dir{<};
      (4,4)*{};(12,4)*{} **\crv{(4,10) & (12,10)}?(1)*\dir{>}?(0)*\dir{>};
      (-20,11)*{\scs j};(-10,11)*{\scs i};
      (20,-11)*{\scs j};(10,-11)*{\scs i};
     }};
  \endxy
\quad =  \quad
\xy 0;/r.17pc/:
  (0,0)*{\xybox{
    (-4,-4)*{};(4,4)*{} **\crv{(-4,-1) & (4,1)}?(1)*\dir{>};
    (4,-4)*{};(-4,4)*{} **\crv{(4,-1) & (-4,1)};
     (4,4)*{};(-18,4)*{} **\crv{(4,16) & (-18,16)} ?(1)*\dir{>};
     (-4,-4)*{};(18,-4)*{} **\crv{(-4,-16) & (18,-16)} ?(1)*\dir{<}?(0)*\dir{<};
     (18,-4);(18,12) **\dir{-};(12,-4);(12,12) **\dir{-};
     (-18,4);(-18,-12) **\dir{-};(-12,4);(-12,-12) **\dir{-};
     (22,1)*{ \lambda};
     (-10,0)*{};(10,0)*{};
      (4,-4)*{};(12,-4)*{} **\crv{(4,-10) & (12,-10)}?(1)*\dir{<}?(0)*\dir{<};
      (-4,4)*{};(-12,4)*{} **\crv{(-4,10) & (-12,10)}?(1)*\dir{>}?(0)*\dir{>};
      (20,11)*{\scs i};(10,11)*{\scs j};
      (-20,-11)*{\scs i};(-10,-11)*{\scs j};
     }};
  \endxy
\end{equation}

Sideways crossings satisfy the following identities:
\begin{equation} \label{eq_crossl-gen-cyc}
  \xy 0;/r.18pc/:
  (0,0)*{\xybox{
    (-4,-4)*{};(4,4)*{} **\crv{(-4,-1) & (4,1)}?(1)*\dir{>} ;
    (4,-4)*{};(-4,4)*{} **\crv{(4,-1) & (-4,1)}?(0)*\dir{<};
    (-5,-3)*{\scs j};
     (6.5,-3)*{\scs i};
     (9,2)*{ \lambda};
     (-12,0)*{};(12,0)*{};
     }};
  \endxy
\quad = \quad
 \xy 0;/r.17pc/:
  (0,0)*{\xybox{
    (4,-4)*{};(-4,4)*{} **\crv{(4,-1) & (-4,1)}?(1)*\dir{>};
    (-4,-4)*{};(4,4)*{} **\crv{(-4,-1) & (4,1)};
     (-4,4);(-4,12) **\dir{-};
     (-12,-4);(-12,12) **\dir{-};
     (4,-4);(4,-12) **\dir{-};(12,4);(12,-12) **\dir{-};
     (16,1)*{\lambda};
     (-10,0)*{};(10,0)*{};
     (-4,-4)*{};(-12,-4)*{} **\crv{(-4,-10) & (-12,-10)}?(1)*\dir{<}?(0)*\dir{<};
      (4,4)*{};(12,4)*{} **\crv{(4,10) & (12,10)}?(1)*\dir{>}?(0)*\dir{>};
      (-14,11)*{\scs i};(-2,11)*{\scs j};
      (14,-11)*{\scs i};(2,-11)*{\scs j};
     }};
  \endxy
  \quad = \quad
 \xy 0;/r.17pc/:
  (0,0)*{\xybox{
    (-4,-4)*{};(4,4)*{} **\crv{(-4,-1) & (4,1)}?(1)*\dir{<};
    (4,-4)*{};(-4,4)*{} **\crv{(4,-1) & (-4,1)};
     (4,4);(4,12) **\dir{-};
     (12,-4);(12,12) **\dir{-};
     (-4,-4);(-4,-12) **\dir{-};(-12,4);(-12,-12) **\dir{-};
     (16,1)*{\lambda};
     (10,0)*{};(-10,0)*{};
     (4,-4)*{};(12,-4)*{} **\crv{(4,-10) & (12,-10)}?(1)*\dir{>}?(0)*\dir{>};
      (-4,4)*{};(-12,4)*{} **\crv{(-4,10) & (-12,10)}?(1)*\dir{<}?(0)*\dir{<};
     }};
     (12,11)*{\scs j};(0,11)*{\scs i};
      (-17,-11)*{\scs j};(-5,-11)*{\scs i};
  \endxy
\end{equation}
\begin{equation} \label{eq_crossr-gen-cyc}
  \xy 0;/r.18pc/:
  (0,0)*{\xybox{
    (-4,-4)*{};(4,4)*{} **\crv{(-4,-1) & (4,1)}?(0)*\dir{<} ;
    (4,-4)*{};(-4,4)*{} **\crv{(4,-1) & (-4,1)}?(1)*\dir{>};
    (5.1,-3)*{\scs i};
     (-6.5,-3)*{\scs j};
     (9,2)*{ \lambda};
     (-12,0)*{};(12,0)*{};
     }};
  \endxy
\quad = \quad
 \xy 0;/r.17pc/:
  (0,0)*{\xybox{
    (-4,-4)*{};(4,4)*{} **\crv{(-4,-1) & (4,1)}?(1)*\dir{>};
    (4,-4)*{};(-4,4)*{} **\crv{(4,-1) & (-4,1)};
     (4,4);(4,12) **\dir{-};
     (12,-4);(12,12) **\dir{-};
     (-4,-4);(-4,-12) **\dir{-};(-12,4);(-12,-12) **\dir{-};
     (16,-6)*{\lambda};
     (10,0)*{};(-10,0)*{};
     (4,-4)*{};(12,-4)*{} **\crv{(4,-10) & (12,-10)}?(1)*\dir{<}?(0)*\dir{<};
      (-4,4)*{};(-12,4)*{} **\crv{(-4,10) & (-12,10)}?(1)*\dir{>}?(0)*\dir{>};
      (14,11)*{\scs j};(2,11)*{\scs i};
      (-14,-11)*{\scs j};(-2,-11)*{\scs i};
     }};
  \endxy
  \quad = \quad
  \xy 0;/r.17pc/:
  (0,0)*{\xybox{
    (4,-4)*{};(-4,4)*{} **\crv{(4,-1) & (-4,1)}?(1)*\dir{<};
    (-4,-4)*{};(4,4)*{} **\crv{(-4,-1) & (4,1)};
     (-4,4);(-4,12) **\dir{-};
     (-12,-4);(-12,12) **\dir{-};
     (4,-4);(4,-12) **\dir{-};(12,4);(12,-12) **\dir{-};
     (16,6)*{\lambda};
     (-10,0)*{};(10,0)*{};
     (-4,-4)*{};(-12,-4)*{} **\crv{(-4,-10) & (-12,-10)}?(1)*\dir{>}?(0)*\dir{>};
      (4,4)*{};(12,4)*{} **\crv{(4,10) & (12,10)}?(1)*\dir{<}?(0)*\dir{<};
      (-14,11)*{\scs i};(-2,11)*{\scs j};(14,-11)*{\scs i};(2,-11)*{\scs j};
     }};
  \endxy
\end{equation}

\item The $\cal{E}$'s (respectively $\cal{F}$'s) carry an action of the KLR algebra associated to $Q$. The KLR algebra $R=R_Q$ associated to $Q$ is defined by finite $\Bbbk$-linear combinations of braid--like diagrams in the plane, where each strand is labeled by a vertex $i \in I$.  Strands can intersect and can carry dots but triple intersections are not allowed.  Diagrams are considered up to planar isotopy that do not change the combinatorial type of the diagram. We recall the local relations:
\begin{enumerate}[i)]

\item For $i \neq j$
\begin{equation}
 \vcenter{\xy 0;/r.17pc/:
    (-4,-4)*{};(4,4)*{} **\crv{(-4,-1) & (4,1)}?(1)*\dir{};
    (4,-4)*{};(-4,4)*{} **\crv{(4,-1) & (-4,1)}?(1)*\dir{};
    (-4,4)*{};(4,12)*{} **\crv{(-4,7) & (4,9)}?(1)*\dir{};
    (4,4)*{};(-4,12)*{} **\crv{(4,7) & (-4,9)}?(1)*\dir{};
    (8,8)*{\lambda};
    (4,12); (4,13) **\dir{-}?(1)*\dir{>};
    (-4,12); (-4,13) **\dir{-}?(1)*\dir{>};
  (-5.5,-3)*{\scs i};
     (5.5,-3)*{\scs j};
 \endxy}
 \qquad = \qquad
 \left\{
 \begin{array}{ccc}
 0 & & \text{if $(\alpha_i,\alpha_j)=2$, } \\ \\
     t_{ij}\;\xy 0;/r.17pc/:
  (3,9);(3,-9) **\dir{-}?(0)*\dir{<}+(2.3,0)*{};
  (-3,9);(-3,-9) **\dir{-}?(0)*\dir{<}+(2.3,0)*{};
  (-5,-6)*{\scs i};     (5.1,-6)*{\scs j};
 \endxy &  &  \text{if $(\alpha_i, \alpha_j)=0$,}\\ \\
  t_{ij} \vcenter{\xy 0;/r.17pc/:
  (3,9);(3,-9) **\dir{-}?(.5)*\dir{<}+(2.3,0)*{};
  (-3,9);(-3,-9) **\dir{-}?(.5)*\dir{<}+(2.3,0)*{};
  (-3,4)*{\bullet};(-6.5,5)*{\scs d_{ij}};
  (-5,-6)*{\scs i};     (5.1,-6)*{\scs j};
 \endxy} \;\; + \;\; t_{ji}
  \vcenter{\xy 0;/r.17pc/:
  (3,9);(3,-9) **\dir{-}?(.5)*\dir{<}+(2.3,0)*{};
  (-3,9);(-3,-9) **\dir{-}?(.5)*\dir{<}+(2.3,0)*{};
  (3,4)*{\bullet};(7,5)*{\scs d_{ji}};
  (-5,-6)*{\scs i};     (5.1,-6)*{\scs j};
 \endxy}
 \;\; + \;\;
 \xsum{p,q}{} s_{ij}^{pq}\;
 \vcenter{\xy 0;/r.17pc/:
  (3,9);(3,-9) **\dir{-}?(.5)*\dir{<}+(2.3,0)*{};
  (-3,9);(-3,-9) **\dir{-}?(.5)*\dir{<}+(2.3,0)*{};
  (-3,4)*{\bullet};(-5.5,5)*{\scs p};
  (3,4)*{\bullet};(6,5)*{\scs q};
  (-5,-6)*{\scs i};     (5.1,-6)*{\scs j};
 \endxy}
   &  & \text{if $(\alpha_i, \alpha_j) <0$,}
 \end{array}
 \right. \label{eq_r2_ij-gen-cyc}
\end{equation}

\item The nilHecke dot sliding relations
\begin{equation} \label{eq_dot_slide_ii-gen-cyc}
\xy 0;/r.18pc/:
  (0,0)*{\xybox{
    (-4,-4)*{};(4,6)*{} **\crv{(-4,-1) & (4,1)}?(1)*\dir{>}?(.25)*{\bullet};
    (4,-4)*{};(-4,6)*{} **\crv{(4,-1) & (-4,1)}?(1)*\dir{>};
    (-5,-3)*{\scs i};
     (5.1,-3)*{\scs i};
     (-10,0)*{};(10,0)*{};
     }};
  \endxy
 \;\; -
\xy 0;/r.18pc/:
  (0,0)*{\xybox{
    (-4,-4)*{};(4,6)*{} **\crv{(-4,-1) & (4,1)}?(1)*\dir{>}?(.75)*{\bullet};
    (4,-4)*{};(-4,6)*{} **\crv{(4,-1) & (-4,1)}?(1)*\dir{>};
    (-5,-3)*{\scs i};
     (5.1,-3)*{\scs i};
     (-10,0)*{};(10,0)*{};
     }};
  \endxy
\;\; =  \xy 0;/r.18pc/:
  (0,0)*{\xybox{
    (-4,-4)*{};(4,6)*{} **\crv{(-4,-1) & (4,1)}?(1)*\dir{>};
    (4,-4)*{};(-4,6)*{} **\crv{(4,-1) & (-4,1)}?(1)*\dir{>}?(.75)*{\bullet};
    (-5,-3)*{\scs i};
     (5.1,-3)*{\scs i};
     (-10,0)*{};(10,0)*{};
     }};
  \endxy
\;\;  -
  \xy 0;/r.18pc/:
  (0,0)*{\xybox{
    (-4,-4)*{};(4,6)*{} **\crv{(-4,-1) & (4,1)}?(1)*\dir{>} ;
    (4,-4)*{};(-4,6)*{} **\crv{(4,-1) & (-4,1)}?(1)*\dir{>}?(.25)*{\bullet};
    (-5,-3)*{\scs i};
     (5.1,-3)*{\scs i};
     (-10,0)*{};(12,0)*{};
     }};
  \endxy
  \;\; = \;\;
    \xy 0;/r.18pc/:
  (0,0)*{\xybox{
    (-4,-4)*{};(-4,6)*{} **\dir{-}?(1)*\dir{>} ;
    (4,-4)*{};(4,6)*{} **\dir{-}?(1)*\dir{>};
    (-5,-3)*{\scs i};
     (5.1,-3)*{\scs i};
     (-10,0)*{};(12,0)*{};
     }};
  \endxy
\end{equation}
hold.

\item For $i \neq j$ the dot sliding relations
\begin{equation} \label{eq_dot_slide_ij-gen-cyc}
\xy 0;/r.18pc/:
  (0,0)*{\xybox{
    (-4,-4)*{};(4,6)*{} **\crv{(-4,-1) & (4,1)}?(1)*\dir{>}?(.75)*{\bullet};
    (4,-4)*{};(-4,6)*{} **\crv{(4,-1) & (-4,1)}?(1)*\dir{>};
    (-5,-3)*{\scs i};
     (5.1,-3)*{\scs j};
     (-10,0)*{};(10,0)*{};
     }};
  \endxy
 \;\; =
\xy 0;/r.18pc/:
  (0,0)*{\xybox{
    (-4,-4)*{};(4,6)*{} **\crv{(-4,-1) & (4,1)}?(1)*\dir{>}?(.25)*{\bullet};
    (4,-4)*{};(-4,6)*{} **\crv{(4,-1) & (-4,1)}?(1)*\dir{>};
    (-5,-3)*{\scs i};
     (5.1,-3)*{\scs j};
     (-10,0)*{};(10,0)*{};
     }};
  \endxy
\qquad  \xy 0;/r.18pc/:
  (0,0)*{\xybox{
    (-4,-4)*{};(4,6)*{} **\crv{(-4,-1) & (4,1)}?(1)*\dir{>};
    (4,-4)*{};(-4,6)*{} **\crv{(4,-1) & (-4,1)}?(1)*\dir{>}?(.75)*{\bullet};
    (-5,-3)*{\scs i};
     (5.1,-3)*{\scs j};
     (-10,0)*{};(10,0)*{};
     }};
  \endxy
\;\;  =
  \xy 0;/r.18pc/:
  (0,0)*{\xybox{
    (-4,-4)*{};(4,6)*{} **\crv{(-4,-1) & (4,1)}?(1)*\dir{>} ;
    (4,-4)*{};(-4,6)*{} **\crv{(4,-1) & (-4,1)}?(1)*\dir{>}?(.25)*{\bullet};
    (-5,-3)*{\scs i};
     (5.1,-3)*{\scs j};
     (-10,0)*{};(12,0)*{};
     }};
  \endxy
\end{equation}
hold.

\item Unless $i = k$ and $(\alpha_i, \alpha_j) < 0$ the relation
\begin{equation}
\vcenter{\xy 0;/r.17pc/:
    (-4,-4)*{};(4,4)*{} **\crv{(-4,-1) & (4,1)}?(1)*\dir{};
    (4,-4)*{};(-4,4)*{} **\crv{(4,-1) & (-4,1)}?(1)*\dir{};
    (4,4)*{};(12,12)*{} **\crv{(4,7) & (12,9)}?(1)*\dir{};
    (12,4)*{};(4,12)*{} **\crv{(12,7) & (4,9)}?(1)*\dir{};
    (-4,12)*{};(4,20)*{} **\crv{(-4,15) & (4,17)}?(1)*\dir{};
    (4,12)*{};(-4,20)*{} **\crv{(4,15) & (-4,17)}?(1)*\dir{};
    (-4,4)*{}; (-4,12) **\dir{-};
    (12,-4)*{}; (12,4) **\dir{-};
    (12,12)*{}; (12,20) **\dir{-};
    (4,20); (4,21) **\dir{-}?(1)*\dir{>};
    (-4,20); (-4,21) **\dir{-}?(1)*\dir{>};
    (12,20); (12,21) **\dir{-}?(1)*\dir{>};
   (18,8)*{\lambda};  (-6,-3)*{\scs i};
  (6,-3)*{\scs j};
  (15,-3)*{\scs k};
\endxy}
 \;\; =\;\;
\vcenter{\xy 0;/r.17pc/:
    (4,-4)*{};(-4,4)*{} **\crv{(4,-1) & (-4,1)}?(1)*\dir{};
    (-4,-4)*{};(4,4)*{} **\crv{(-4,-1) & (4,1)}?(1)*\dir{};
    (-4,4)*{};(-12,12)*{} **\crv{(-4,7) & (-12,9)}?(1)*\dir{};
    (-12,4)*{};(-4,12)*{} **\crv{(-12,7) & (-4,9)}?(1)*\dir{};
    (4,12)*{};(-4,20)*{} **\crv{(4,15) & (-4,17)}?(1)*\dir{};
    (-4,12)*{};(4,20)*{} **\crv{(-4,15) & (4,17)}?(1)*\dir{};
    (4,4)*{}; (4,12) **\dir{-};
    (-12,-4)*{}; (-12,4) **\dir{-};
    (-12,12)*{}; (-12,20) **\dir{-};
    (4,20); (4,21) **\dir{-}?(1)*\dir{>};
    (-4,20); (-4,21) **\dir{-}?(1)*\dir{>};
    (-12,20); (-12,21) **\dir{-}?(1)*\dir{>};
  (10,8)*{\lambda};
  (-14,-3)*{\scs i};
  (-6,-3)*{\scs j};
  (6,-3)*{\scs k};
\endxy}
 \label{eq_r3_easy-gen-cyc}
\end{equation}
holds. Otherwise, $(\alpha_i, \alpha_j) < 0$ and
\begin{equation}
\;\; \vcenter{
 \xy 0;/r.17pc/:
    (-4,-4)*{};(4,4)*{} **\crv{(-4,-1) & (4,1)}?(1)*\dir{};
    (4,-4)*{};(-4,4)*{} **\crv{(4,-1) & (-4,1)}?(1)*\dir{};
    (4,4)*{};(12,12)*{} **\crv{(4,7) & (12,9)}?(1)*\dir{};
    (12,4)*{};(4,12)*{} **\crv{(12,7) & (4,9)}?(1)*\dir{};
    (-4,12)*{};(4,20)*{} **\crv{(-4,15) & (4,17)}?(1)*\dir{>};
    (4,12)*{};(-4,20)*{} **\crv{(4,15) & (-4,17)}?(1)*\dir{>};
    (-4,4)*{}; (-4,12) **\dir{-};
    (12,-4)*{}; (12,4) **\dir{-};
    (12,12)*{}; (12,20) **\dir{-}?(1)*\dir{>};
  (-6,-3)*{\scs i};
  (6,-3)*{\scs j};
  (14,-3)*{\scs i};
\endxy}
\quad - \quad
 \vcenter{
 \xy 0;/r.17pc/:
    (4,-4)*{};(-4,4)*{} **\crv{(4,-1) & (-4,1)}?(1)*\dir{};
    (-4,-4)*{};(4,4)*{} **\crv{(-4,-1) & (4,1)}?(1)*\dir{};
    (-4,4)*{};(-12,12)*{} **\crv{(-4,7) & (-12,9)}?(1)*\dir{};
    (-12,4)*{};(-4,12)*{} **\crv{(-12,7) & (-4,9)}?(1)*\dir{};
    (4,12)*{};(-4,20)*{} **\crv{(4,15) & (-4,17)}?(1)*\dir{>};
    (-4,12)*{};(4,20)*{} **\crv{(-4,15) & (4,17)}?(1)*\dir{>};
    (4,4)*{}; (4,12) **\dir{-};
    (-12,-4)*{}; (-12,4) **\dir{-};
    (-12,12)*{}; (-12,20) **\dir{-}?(1)*\dir{>};
  (6,-3)*{\scs i};
  (-6,-3)*{\scs j};
  (-14,-3)*{\scs i};
\endxy}\;\;
 \;\; =\;\;
 t_{ij}\sum_{\ell_1+\ell_2=d_{ij}-1} \;\;
\xy 0;/r.17pc/:
  (4,12);(4,-12) **\dir{-}?(.5)*\dir{<};
  (-4,12);(-4,-12) **\dir{-}?(.5)*\dir{<}?(.25)*\dir{>}+(0,0)*{\bullet}+(-3,0)*{\scs \ell_1};
  (12,12);(12,-12) **\dir{-}?(.5)*\dir{<}?(.25)*\dir{>}+(0,0)*{\bullet}+(3,0)*{\scs \ell_2};
  (-6,-9)*{\scs i};     (6.1,-9)*{\scs j};
  (14,-9)*{\scs i};
 \endxy
 \;\; + \;\;
 \sum_{p,q} s_{ij}^{pq} \sum_{\xy (0,2.5)*{\scs \ell_1+\ell_2}; (0,-1)*{\scs =p-1}; \endxy}
 \xy 0;/r.17pc/:
  (4,12);(4,-12)   **\dir{-}?(.5)*\dir{<} ?(.25)*\dir{}+(0,0)*{\bullet}+(-3,0)*{\scs q};
  (-4,12);(-4,-12) **\dir{-}?(.5)*\dir{<}?(.25)*\dir{}+(0,0)*{\bullet}+(-3,0)*{\scs \ell_1};
  (12,12);(12,-12) **\dir{-}?(.5)*\dir{<}?(.25)*\dir{}+(0,0)*{\bullet}+(3,0)*{\scs \ell_2};
  (-6,-9)*{\scs i};     (6.1,-9)*{\scs j};
  (14,-9)*{\scs i};
 \endxy
 \label{eq_r3_hard-gen-cyc}
\end{equation}
\end{enumerate}

\item When $i \ne j$ one has the mixed relations  relating $\cal{E}_i \cal{F}_j$ and $\cal{F}_j \cal{E}_i$:
\begin{equation} \label{mixed_rel-cyc}
 \vcenter{   \xy 0;/r.18pc/:
    (-4,-4)*{};(4,4)*{} **\crv{(-4,-1) & (4,1)}?(1)*\dir{>};
    (4,-4)*{};(-4,4)*{} **\crv{(4,-1) & (-4,1)}?(1)*\dir{<};?(0)*\dir{<};
    (-4,4)*{};(4,12)*{} **\crv{(-4,7) & (4,9)};
    (4,4)*{};(-4,12)*{} **\crv{(4,7) & (-4,9)}?(1)*\dir{>};
  (8,8)*{\lambda};(-6,-3)*{\scs i};
     (6,-3)*{\scs j};
 \endxy}
 \;\; = \;\;
\xy 0;/r.18pc/:
  (3,9);(3,-9) **\dir{-}?(.55)*\dir{>}+(2.3,0)*{};
  (-3,9);(-3,-9) **\dir{-}?(.5)*\dir{<}+(2.3,0)*{};
  (8,2)*{\lambda};(-5,-6)*{\scs i};     (5.1,-6)*{\scs j};
 \endxy
\qquad \quad
    \vcenter{\xy 0;/r.18pc/:
    (-4,-4)*{};(4,4)*{} **\crv{(-4,-1) & (4,1)}?(1)*\dir{<};?(0)*\dir{<};
    (4,-4)*{};(-4,4)*{} **\crv{(4,-1) & (-4,1)}?(1)*\dir{>};
    (-4,4)*{};(4,12)*{} **\crv{(-4,7) & (4,9)}?(1)*\dir{>};
    (4,4)*{};(-4,12)*{} **\crv{(4,7) & (-4,9)};
  (8,8)*{\lambda};(-6,-3)*{\scs i};
     (6,-3)*{\scs j};
 \endxy}
 \;\;=\;\;
\xy 0;/r.18pc/:
  (3,9);(3,-9) **\dir{-}?(.5)*\dir{<}+(2.3,0)*{};
  (-3,9);(-3,-9) **\dir{-}?(.55)*\dir{>}+(2.3,0)*{};
  (8,2)*{\lambda};(-5,-6)*{\scs i};     (5.1,-6)*{\scs j};
 \endxy
\end{equation}

\item \label{item_positivity-cyc} Negative degree bubbles are zero. That is, for all $m \in \Z_+$ one has
\begin{equation} \label{eq_positivity_bubbles-cyc}
\xy 0;/r.18pc/:
 (-12,0)*{\icbub{m}{i}};
 (-8,8)*{\lambda};
 \endxy
  = 0
 \qquad  \text{if $m< \la i,\lambda\ra-1$,} \qquad \xy 0;/r.18pc/: (-12,0)*{\iccbub{m}{i}};
 (-8,8)*{\lambda};
 \endxy = 0\quad
  \text{if $m< -\la i,\lambda\ra-1$.}
\end{equation}
On the other hand, a dotted bubble of degree zero is a scalar multiple
of the identity 2-morphism:
\[
\xy 0;/r.18pc/:
 (0,0)*{\icbub{\la i,\lambda\ra-1}{i}};
  (4,8)*{\lambda};
 \endxy
  =  c_{i,\l} \Id_{\onenn{\lambda}} \quad \text{for $\la i,\lambda\ra \geq 1$,}
  \qquad \quad
  \xy 0;/r.18pc/:
 (0,0)*{\iccbub{-\la i,\lambda\ra-1}{i}};
  (4,8)*{\lambda};
 \endxy  =  c_{i,\l}^{-1} \Id_{\onenn{\lambda}} \quad \text{for $\la i,\lambda\ra \leq -1$.}\]

\item \label{item_highersl2-cyc} For any $i \in I$ one has the extended ${\mathfrak{sl}}_2$-relations. In order to describe certain extended ${\mathfrak{sl}}_2$ relations it is convenient to use a shorthand notation from \cite{Lau1} called fake bubbles. These are diagrams for dotted bubbles where the labels of the number of dots is negative, but the total degree of the dotted bubble taken with these negative dots is still positive. They allow us to write these extended ${\mathfrak{sl}}_2$ relations more uniformly (i.e. independent on whether the weight $\la i,\lambda\ra$ is positive or negative).
\begin{itemize}
 \item Degree zero fake bubbles are equal to
\[
 \xy 0;/r.18pc/:
    (2,0)*{\icbub{\la i,\lambda\ra-1}{i}};
  (12,8)*{\lambda};
 \endxy =
 c_{i,\l}\Id_{\onenn{\lambda}} \quad \text{if $\la i,\lambda\ra \leq 0$,}
  \qquad \quad
\xy 0;/r.18pc/:
    (2,0)*{\iccbub{-\la i,\lambda\ra-1}{i}};
  (12,8)*{\lambda};
 \endxy =   c_{i,\l}^{-1}\Id_{\onenn{\lambda}} \quad  \text{if $\la i,\lambda\ra \geq 0$}.\]

  \item Higher degree fake bubbles for $\la i,\lambda\ra<0$ are defined inductively as
  \begin{equation} \label{eq_fake_nleqz-cyc}
  \vcenter{\xy 0;/r.18pc/:
    (2,-11)*{\icbub{\la i,\lambda\ra-1+j}{i}};
  (12,-2)*{\l};
 \endxy} \;\; =
 \left\{
 \begin{array}{cl}
  \;\; -\;\; c_{i,\l}
\xsum{\xy (0,6)*{};  (0,1)*{\scs a+b=j}; (0,-2)*{\scs b\geq 1}; \endxy}
\;\; \vcenter{\xy 0;/r.18pc/:
    (2,0)*{\cbub{\la i,\lambda\ra-1+a}{}};
    (24,0)*{\ccbub{-\la i,\lambda\ra-1+b}{}};
  (12,8)*{\lambda};
 \endxy}  & \text{if $0 < j < -\la i,\lambda\ra+1$} \\ & \\
   0 & \text{if $j < 0$. }
 \end{array}
\right.
 \end{equation}

  \item Higher degree fake bubbles for $\la i,\lambda\ra>0$ are defined inductively as
   \begin{equation} \label{eq_fake_ngeqz-cyc}
  \vcenter{\xy 0;/r.18pc/:
    (2,-11)*{\iccbub{-\la i,\lambda\ra-1+j}{i}};
  (12,-2)*{\l};
 \endxy} \;\; =
 \left\{
 \begin{array}{cl}
  \;\; -\;\;c_{i,\l}^{-1}
\xsum{\xy (0,6)*{}; (0,1)*{\scs a+b=j}; (0,-2)*{\scs a\geq 1}; \endxy}
\;\; \vcenter{\xy 0;/r.18pc/:
    (2,0)*{\cbub{\la i,\lambda\ra-1+a}{}};
    (24,0)*{\ccbub{-\la i,\lambda\ra-1+b}{}};
  (12,8)*{\lambda};
 \endxy}  & \text{if $0 < j < \la i,\lambda\ra+1$} \\ & \\
   0 & \text{if $j < 0$. }
 \end{array}
\right.
\end{equation}
\end{itemize}
These equations arise from the homogeneous terms in $t$ of the `infinite Grassmannian' equation
\begin{center}
\begin{eqnarray}
 \makebox[0pt]{ $
\left( \xy 0;/r.15pc/:
 (0,0)*{\iccbub{-\la i,\lambda\ra-1}{i}};
  (4,8)*{\l};
 \endxy
 +
 \xy 0;/r.15pc/:
 (0,0)*{\iccbub{-\la i,\lambda\ra-1+1}{i}};
  (4,8)*{\l};
 \endxy t
 + \cdots +
\xy 0;/r.15pc/:
 (0,0)*{\iccbub{-\la i,\lambda\ra-1+\alpha}{i}};
  (4,8)*{\l};
 \endxy t^{\alpha}
 + \cdots
\right)
\left(\xy 0;/r.15pc/:
 (0,0)*{\icbub{\la i,\lambda\ra-1}{i}};
  (4,8)*{\l};
 \endxy
 + \xy 0;/r.15pc/:
 (0,0)*{\icbub{\la i,\lambda\ra-1+1}{i}};
  (4,8)*{\l};
 \endxy t
 +\cdots +
\xy 0;/r.15pc/:
 (0,0)*{\icbub{\la i,\lambda\ra-1+\alpha}{i}};
 (4,8)*{\l};
 \endxy t^{\alpha}
 + \cdots
\right) = \Id_{\onel}.$ } \nn \\ \label{eq_infinite_Grass-cyc}
\end{eqnarray}
\end{center}
Now we can define the extended ${\mathfrak{sl}}_2$ relations.  Note that in \cite{CLau} additional curl relations were provided that can be derived from those above.
For the following relations we employ the convention that all
summations are nonnegative, so that $\sum_{f_1+\cdots+f_n=\alpha}$ is
zero if $\alpha<0$.
\begin{equation} \label{eq_EF-cyc}
 \vcenter{\xy 0;/r.17pc/:
  (-8,0)*{};
  (8,0)*{};
  (-4,10)*{}="t1";
  (4,10)*{}="t2";
  (-4,-10)*{}="b1";
  (4,-10)*{}="b2";(-6,-8)*{\scs i};(6,-8)*{\scs i};
  "t1";"b1" **\dir{-} ?(.5)*\dir{<};
  "t2";"b2" **\dir{-} ?(.5)*\dir{>};
  (10,2)*{\l};
  \endxy}
\;\; = \;\; -\;\;
 \vcenter{   \xy 0;/r.17pc/:
    (-4,-4)*{};(4,4)*{} **\crv{(-4,-1) & (4,1)}?(1)*\dir{>};
    (4,-4)*{};(-4,4)*{} **\crv{(4,-1) & (-4,1)}?(1)*\dir{<};?(0)*\dir{<};
    (-4,4)*{};(4,12)*{} **\crv{(-4,7) & (4,9)};
    (4,4)*{};(-4,12)*{} **\crv{(4,7) & (-4,9)}?(1)*\dir{>};
  (8,8)*{\l};
     (-6,-3)*{\scs i};
     (6.5,-3)*{\scs i};
 \endxy}
  \;\; + \;\;
   \sum_{ \xy  (0,3)*{\scs f_1+f_2+f_3}; (0,0)*{\scs =\la i,\lambda\ra-1};\endxy}
    \vcenter{\xy 0;/r.17pc/:
    (-12,10)*{\l};
    (-8,0)*{};
  (8,0)*{};
  (-4,-15)*{}="b1";
  (4,-15)*{}="b2";
  "b2";"b1" **\crv{(5,-8) & (-5,-8)}; ?(.05)*\dir{<} ?(.93)*\dir{<}
  ?(.8)*\dir{}+(0,-.1)*{\bullet}+(-3,2)*{\scs f_3};
  (-4,15)*{}="t1";
  (4,15)*{}="t2";
  "t2";"t1" **\crv{(5,8) & (-5,8)}; ?(.15)*\dir{>} ?(.95)*\dir{>}
  ?(.4)*\dir{}+(0,-.2)*{\bullet}+(3,-2)*{\scs \; f_1};
  (0,0)*{\iccbub{\scs \quad -\la i,\lambda\ra-1+f_2}{i}};
  (7,-13)*{\scs i};
  (-7,13)*{\scs i};
  \endxy}
\end{equation}
\begin{equation}\label{eq_FE-cyc}
 \vcenter{\xy 0;/r.17pc/:
  (-8,0)*{};(-6,-8)*{\scs i};(6,-8)*{\scs i};
  (8,0)*{};
  (-4,10)*{}="t1";
  (4,10)*{}="t2";
  (-4,-10)*{}="b1";
  (4,-10)*{}="b2";
  "t1";"b1" **\dir{-} ?(.5)*\dir{>};
  "t2";"b2" **\dir{-} ?(.5)*\dir{<};
  (10,2)*{\l};
  (-10,2)*{\l};
  \endxy}
\;\; = \;\;
  -\;\;\vcenter{\xy 0;/r.17pc/:
    (-4,-4)*{};(4,4)*{} **\crv{(-4,-1) & (4,1)}?(1)*\dir{<};?(0)*\dir{<};
    (4,-4)*{};(-4,4)*{} **\crv{(4,-1) & (-4,1)}?(1)*\dir{>};
    (-4,4)*{};(4,12)*{} **\crv{(-4,7) & (4,9)}?(1)*\dir{>};
    (4,4)*{};(-4,12)*{} **\crv{(4,7) & (-4,9)};
  (8,8)*{\l};(-6.5,-3)*{\scs i};  (6,-3)*{\scs i};
 \endxy}
  \;\; + \;\;
    \sum_{ \xy  (0,3)*{\scs g_1+g_2+g_3}; (0,0)*{\scs =-\la i,\lambda\ra-1};\endxy}
    \vcenter{\xy 0;/r.17pc/:
    (-8,0)*{};
  (8,0)*{};
  (-4,-15)*{}="b1";
  (4,-15)*{}="b2";
  "b2";"b1" **\crv{(5,-8) & (-5,-8)}; ?(.1)*\dir{>} ?(.95)*\dir{>}
  ?(.8)*\dir{}+(0,-.1)*{\bullet}+(-3,2)*{\scs g_3};
  (-4,15)*{}="t1";
  (4,15)*{}="t2";
  "t2";"t1" **\crv{(5,8) & (-5,8)}; ?(.15)*\dir{<} ?(.9)*\dir{<}
  ?(.4)*\dir{}+(0,-.2)*{\bullet}+(3,-2)*{\scs g_1};
  (0,0)*{\icbub{\scs \quad\; \la i,\lambda\ra-1 + g_2}{i}};
    (7,-13)*{\scs i};
  (-7,13)*{\scs i};
  (-10,10)*{\l};
  \endxy}
\end{equation}
\end{enumerate}

\subsection{A duality functor}
\label{sec:duality-functor}

As discussed in the introduction, our motivation for considering
this modified diagrammatic framework is in order to construct a cyclic
duality.  This functor arises from the ability to
rotate diagrams.  It is defined as in \cite[3.46]{KL3} by taking the
usual adjoint, using the adjoints defined by cups and caps.  We refer
to that source for the effect this duality has one 1-morphisms, since
this is the same in $\Ucat_Q(\mf{g})$ and
$\Ucat_Q^{cyc}(\mf{g})$. Most importantly $\sE_i$ and $\sF_i$ are
dual.  On
the level of 2-morphisms, our duality is defined by
\begin{equation}
 \tikz[baseline,very thick]{\draw (0,-1) --
  node[midway,draw=black,fill=white]{$f^*$}
  (0,1);}:=\tikz[baseline,very
thick]{\node[draw=black,fill=white] (a) at (0,0) {$f$}; \draw (-.5,-1) to[out=90,in=180]
  (-.25,.6) to[out=0,in=90] (a.90);\draw
(a.-90)  to[out=-90,in=180]
  (.25,-.6) to[out=0,in=-90] (.5,1);}= \tikz[baseline,very
thick]{\node[draw=black,fill=white] (a) at (0,0) {$f$}; \draw (.5,-1) to[out=90,in=0]
  (.25,.6) to[out=180,in=90] (a.90);\draw
(a.-90)  to[out=-90,in=0]
  (-.25,-.6) to[out=180,in=-90] (-.5,1);}\label{eq:duality}
\end{equation}
The second equality holds by applying (\ref{eq_cyclic_dot-cyc},\ref{eq_cyclic}).
Here, it is very important that the diagrams are interpreted in the
category $\Ucat_Q^{cyc}(\mf{g})$, since
 cyclicity is precisely the property that these two operations give
 the same duality.  The second equality of \eqref{eq:duality} would not hold if these diagrams were
 interpreted in $\Ucat_Q(\mf{g})$.

%
\section{An isomorphism of 2-categories}
\label{sec:iso}
%

\begin{thm} \label{thm}
There is an isomorphism of 2-categories
\[
 \cal{M} \maps \Ucat_Q^{cyc}(\mf{g}) \longrightarrow \Ucat_Q(\mf{g})
\]
where  $\Ucat_Q(\mf{g})$ is the 2-category defined in \cite{CLau}.
\end{thm}

Note here that $\Ucat_Q(\mf{g})$, defined when $\Bbbk$ is a field in
\cite{CLau}, can be extended to the case $\Bbbk$ is a commutative ring
with unit.

\begin{proof}
Define the 2-functor $\cal{M}$ to act as the identity on objects and 1-morphisms
and by rescaling the following generators:
\begin{align}
\cal{M}\left(
\xy 0;/r.17pc/:
    (0,0)*{\bbcfe{i}};
    (8,4)*{\scs  \lambda};
    (-10,0)*{};(10,0)*{};
    \endxy \;\; \right)
& \;\; := \;\; c_{i,\lambda} \cdot
    \hspace{-2mm}\xy 0;/r.17pc/:
    (0,0)*{\bbcfe{i}};
    (8,4)*{\scs  \lambda};
    (-10,0)*{};(10,0)*{};
    \endxy   \nn \\
\cal{M}\left(\vcenter{\xy 0;/r.17pc/:
    (0,-3)*{\bbpef{i}};
    (8,-5)*{\scs \lambda};
    (-10,0)*{};(10,0)*{};
    \endxy}  \right)
 &\;\; := \;\;
   c_{i,\lambda}^{-1} \cdot \hspace{-2mm} \vcenter{\xy 0;/r.17pc/:
    (0,-3)*{\bbpef{i}};
    (8,-5)*{\scs \lambda};
    (-10,0)*{};(10,0)*{};
    \endxy}  \nn \\
 \cal{M}\left(\xy 0;/r.17pc/:
  (0,0)*{\xybox{
    (-4,4)*{};(4,-4)*{} **\crv{(-4,1) & (4,-1)}?(1)*\dir{>} ;
    (4,4)*{};(-4,-4)*{} **\crv{(4,1) & (-4,-1)}?(1)*\dir{>};
    (-6.5,-3)*{\scs i};
     (6.5,-3)*{\scs j};
     (9,1)*{\scs  \lambda};
     (-10,0)*{};(10,0)*{};
     }};
  \endxy \right)
 &   \;\; := \;\; t_{ji}  \;
   \xy 0;/r.17pc/:
  (0,0)*{\xybox{
    (-4,4)*{};(4,-4)*{} **\crv{(-4,1) & (4,-1)}?(1)*\dir{>} ;
    (4,4)*{};(-4,-4)*{} **\crv{(4,1) & (-4,-1)}?(1)*\dir{>};
    (-6.5,-3)*{\scs i};
     (6.5,-3)*{\scs j};
     (9,1)*{\scs  \lambda};
     (-10,0)*{};(10,0)*{};
     }};
  \endxy \nn \\
   \cal{M}\left(  \xy 0;/r.17pc/:
  (0,0)*{\xybox{
    (-4,-4)*{};(4,4)*{} **\crv{(-4,-1) & (4,1)}?(1)*\dir{>} ;
    (4,-4)*{};(-4,4)*{} **\crv{(4,-1) & (-4,1)}?(0)*\dir{<};
    (-5,-3)*{\scs j};
     (6.5,-3)*{\scs i};
     (9,2)*{ \lambda};
     (-12,0)*{};(12,0)*{};
     }};
  \endxy \right)
&  \;\; := \;\;  t_{ij}^{-1}   \;
    \xy 0;/r.17pc/:
  (0,0)*{\xybox{
    (-4,-4)*{};(4,4)*{} **\crv{(-4,-1) & (4,1)}?(1)*\dir{>} ;
    (4,-4)*{};(-4,4)*{} **\crv{(4,-1) & (-4,1)}?(0)*\dir{<};
    (-5,-3)*{\scs j};
     (6.5,-3)*{\scs i};
     (9,2)*{ \lambda};
     (-12,0)*{};(12,0)*{};
     }};
  \endxy
\end{align}

The map $\cal{M}$ preserves relations \eqref{eq_biadjoint1-cyc} -- \eqref{eq_cyclic_dot-cyc}.
The cyclic relations \eqref{eq_cyclic} become
\begin{equation}
  \label{label1}
 t_{ji}  \xy 0;/r.17pc/:
  (0,0)*{\xybox{
    (-4,4)*{};(4,-4)*{} **\crv{(-4,1) & (4,-1)}?(1)*\dir{>} ;
    (4,4)*{};(-4,-4)*{} **\crv{(4,1) & (-4,-1)}?(1)*\dir{>};
    (-6.5,-3)*{\scs i};
     (6.5,-3)*{\scs j};
     (9,1)*{\scs  \lambda};
     (-10,0)*{};(10,0)*{};
     }};
  \endxy \quad = \quad
 t_{ij}^{-1}t_{ji} \xy 0;/r.17pc/:
  (0,0)*{\xybox{
    (4,-4)*{};(-4,4)*{} **\crv{(4,-1) & (-4,1)}?(1)*\dir{>};
    (-4,-4)*{};(4,4)*{} **\crv{(-4,-1) & (4,1)};
     (-4,4)*{};(18,4)*{} **\crv{(-4,16) & (18,16)} ?(1)*\dir{>};
     (4,-4)*{};(-18,-4)*{} **\crv{(4,-16) & (-18,-16)} ?(1)*\dir{<}?(0)*\dir{<};
     (-18,-4);(-18,12) **\dir{-};(-12,-4);(-12,12) **\dir{-};
     (18,4);(18,-12) **\dir{-};(12,4);(12,-12) **\dir{-};
     (22,1)*{ \lambda};
     (-10,0)*{};(10,0)*{};
     (-4,-4)*{};(-12,-4)*{} **\crv{(-4,-10) & (-12,-10)}?(1)*\dir{<}?(0)*\dir{<};
      (4,4)*{};(12,4)*{} **\crv{(4,10) & (12,10)}?(1)*\dir{>}?(0)*\dir{>};
      (-20,11)*{\scs j};(-10,11)*{\scs i};
      (20,-11)*{\scs j};(10,-11)*{\scs i};
     }};
  \endxy
\quad =  \quad
\xy 0;/r.17pc/:
  (0,0)*{\xybox{
    (-4,-4)*{};(4,4)*{} **\crv{(-4,-1) & (4,1)}?(1)*\dir{>};
    (4,-4)*{};(-4,4)*{} **\crv{(4,-1) & (-4,1)};
     (4,4)*{};(-18,4)*{} **\crv{(4,16) & (-18,16)} ?(1)*\dir{>};
     (-4,-4)*{};(18,-4)*{} **\crv{(-4,-16) & (18,-16)} ?(1)*\dir{<}?(0)*\dir{<};
     (18,-4);(18,12) **\dir{-};(12,-4);(12,12) **\dir{-};
     (-18,4);(-18,-12) **\dir{-};(-12,4);(-12,-12) **\dir{-};
     (22,1)*{ \lambda};
     (-10,0)*{};(10,0)*{};
      (4,-4)*{};(12,-4)*{} **\crv{(4,-10) & (12,-10)}?(1)*\dir{<}?(0)*\dir{<};
      (-4,4)*{};(-12,4)*{} **\crv{(-4,10) & (-12,10)}?(1)*\dir{>}?(0)*\dir{>};
      (20,11)*{\scs i};(10,11)*{\scs j};
      (-20,-11)*{\scs i};(-10,-11)*{\scs j};
     }};
  \endxy
\end{equation}
since $c_{j,\l-\alpha_i-\alpha_j}^{-1}c_{i,\l-\alpha_i}^{-1} c_{i,\l-\alpha_j} c_{j,\l}
= \frac{c_{i,\l-\alpha_j}}{c_{i,\l-\alpha_i}}
\frac{c_{j,\l}}{c_{j,\l-\alpha_i-\alpha_j}}=t_{ij}^{-1}t_{ji}$.

Likewise, since $c_{i,\l} c_{i,\l-\alpha_i+\alpha_j}^{-1} = c_{i,\l}/c_{i,\l+\alpha_j}  = t_{ij}^{-1}$ and $c_{j,\l}^{-1} c_{j,\l+\alpha_j-\alpha_i} = c_{j,\l+\alpha_j-\alpha_i}/ c_{j,\l+\alpha_j} =t_{ji}^{-1}$, equations \eqref{eq_crossl-gen-cyc} and \eqref{eq_crossr-gen-cyc} become
\begin{equation}
  t_{ij}^{-1} \;\; \xy 0;/r.18pc/:
  (0,0)*{\xybox{
    (-4,-4)*{};(4,4)*{} **\crv{(-4,-1) & (4,1)}?(1)*\dir{>} ;
    (4,-4)*{};(-4,4)*{} **\crv{(4,-1) & (-4,1)}?(0)*\dir{<};
    (-5,-3)*{\scs j};
     (6.5,-3)*{\scs i};
     (9,2)*{ \lambda};
     (-12,0)*{};(12,0)*{};
     }};
  \endxy
\quad = \quad t_{ij}^{-1}
 \xy 0;/r.17pc/:
  (0,0)*{\xybox{
    (4,-4)*{};(-4,4)*{} **\crv{(4,-1) & (-4,1)}?(1)*\dir{>};
    (-4,-4)*{};(4,4)*{} **\crv{(-4,-1) & (4,1)};
     (-4,4);(-4,12) **\dir{-};
     (-12,-4);(-12,12) **\dir{-};
     (4,-4);(4,-12) **\dir{-};(12,4);(12,-12) **\dir{-};
     (16,1)*{\lambda};
     (-10,0)*{};(10,0)*{};
     (-4,-4)*{};(-12,-4)*{} **\crv{(-4,-10) & (-12,-10)}?(1)*\dir{<}?(0)*\dir{<};
      (4,4)*{};(12,4)*{} **\crv{(4,10) & (12,10)}?(1)*\dir{>}?(0)*\dir{>};
      (-14,11)*{\scs i};(-2,11)*{\scs j};
      (14,-11)*{\scs i};(2,-11)*{\scs j};
     }};
  \endxy
  \quad = \quad
 \xy 0;/r.17pc/:
  (0,0)*{\xybox{
    (-4,-4)*{};(4,4)*{} **\crv{(-4,-1) & (4,1)}?(1)*\dir{<};
    (4,-4)*{};(-4,4)*{} **\crv{(4,-1) & (-4,1)};
     (4,4);(4,12) **\dir{-};
     (12,-4);(12,12) **\dir{-};
     (-4,-4);(-4,-12) **\dir{-};(-12,4);(-12,-12) **\dir{-};
     (16,1)*{\lambda};
     (10,0)*{};(-10,0)*{};
     (4,-4)*{};(12,-4)*{} **\crv{(4,-10) & (12,-10)}?(1)*\dir{>}?(0)*\dir{>};
      (-4,4)*{};(-12,4)*{} **\crv{(-4,10) & (-12,10)}?(1)*\dir{<}?(0)*\dir{<};
     }};
     (12,11)*{\scs j};(0,11)*{\scs i};
      (-17,-11)*{\scs j};(-5,-11)*{\scs i};
  \endxy
\end{equation}
\begin{equation}
  \xy 0;/r.18pc/:
  (0,0)*{\xybox{
    (-4,-4)*{};(4,4)*{} **\crv{(-4,-1) & (4,1)}?(0)*\dir{<} ;
    (4,-4)*{};(-4,4)*{} **\crv{(4,-1) & (-4,1)}?(1)*\dir{>};
    (5.1,-3)*{\scs i};
     (-6.5,-3)*{\scs j};
     (9,2)*{ \lambda};
     (-12,0)*{};(12,0)*{};
     }};
  \endxy
\quad = \quad
 \xy 0;/r.17pc/:
  (0,0)*{\xybox{
    (-4,-4)*{};(4,4)*{} **\crv{(-4,-1) & (4,1)}?(1)*\dir{>};
    (4,-4)*{};(-4,4)*{} **\crv{(4,-1) & (-4,1)};
     (4,4);(4,12) **\dir{-};
     (12,-4);(12,12) **\dir{-};
     (-4,-4);(-4,-12) **\dir{-};(-12,4);(-12,-12) **\dir{-};
     (16,-6)*{\lambda};
     (10,0)*{};(-10,0)*{};
     (4,-4)*{};(12,-4)*{} **\crv{(4,-10) & (12,-10)}?(1)*\dir{<}?(0)*\dir{<};
      (-4,4)*{};(-12,4)*{} **\crv{(-4,10) & (-12,10)}?(1)*\dir{>}?(0)*\dir{>};
      (14,11)*{\scs j};(2,11)*{\scs i};
      (-14,-11)*{\scs j};(-2,-11)*{\scs i};
     }};
  \endxy
  \quad = \quad t_{ji} \;\;
  \xy 0;/r.17pc/:
  (0,0)*{\xybox{
    (4,-4)*{};(-4,4)*{} **\crv{(4,-1) & (-4,1)}?(1)*\dir{<};
    (-4,-4)*{};(4,4)*{} **\crv{(-4,-1) & (4,1)};
     (-4,4);(-4,12) **\dir{-};
     (-12,-4);(-12,12) **\dir{-};
     (4,-4);(4,-12) **\dir{-};(12,4);(12,-12) **\dir{-};
     (16,6)*{\lambda};
     (-10,0)*{};(10,0)*{};
     (-4,-4)*{};(-12,-4)*{} **\crv{(-4,-10) & (-12,-10)}?(1)*\dir{>}?(0)*\dir{>};
      (4,4)*{};(12,4)*{} **\crv{(4,10) & (12,10)}?(1)*\dir{<}?(0)*\dir{<};
      (-14,11)*{\scs i};(-2,11)*{\scs j};(14,-11)*{\scs i};(2,-11)*{\scs j};
     }};
  \endxy
\end{equation}

All of the KLR relations are preserved by $\cal{M}$ and the mixed relation \eqref{mixed_rel-cyc} for $i \neq j$ becomes
\begin{equation}
t_{ji}^{-1} \vcenter{   \xy 0;/r.18pc/:
    (-4,-4)*{};(4,4)*{} **\crv{(-4,-1) & (4,1)}?(1)*\dir{>};
    (4,-4)*{};(-4,4)*{} **\crv{(4,-1) & (-4,1)}?(1)*\dir{<};?(0)*\dir{<};
    (-4,4)*{};(4,12)*{} **\crv{(-4,7) & (4,9)};
    (4,4)*{};(-4,12)*{} **\crv{(4,7) & (-4,9)}?(1)*\dir{>};
  (8,8)*{\lambda};(-6,-3)*{\scs i};
     (6,-3)*{\scs j};
 \endxy}
 \;\; = \;\;
\xy 0;/r.18pc/:
  (3,9);(3,-9) **\dir{-}?(.55)*\dir{>}+(2.3,0)*{};
  (-3,9);(-3,-9) **\dir{-}?(.5)*\dir{<}+(2.3,0)*{};
  (8,2)*{\lambda};(-5,-6)*{\scs i};     (5.1,-6)*{\scs j};
 \endxy
\qquad \quad
  t_{ij}^{-1}  \vcenter{\xy 0;/r.18pc/:
    (-4,-4)*{};(4,4)*{} **\crv{(-4,-1) & (4,1)}?(1)*\dir{<};?(0)*\dir{<};
    (4,-4)*{};(-4,4)*{} **\crv{(4,-1) & (-4,1)}?(1)*\dir{>};
    (-4,4)*{};(4,12)*{} **\crv{(-4,7) & (4,9)}?(1)*\dir{>};
    (4,4)*{};(-4,12)*{} **\crv{(4,7) & (-4,9)};
  (8,8)*{\lambda};(-6,-3)*{\scs i};
     (6,-3)*{\scs j};
 \endxy}
 \;\;=\;\;
\xy 0;/r.18pc/:
  (3,9);(3,-9) **\dir{-}?(.5)*\dir{<}+(2.3,0)*{};
  (-3,9);(-3,-9) **\dir{-}?(.55)*\dir{>}+(2.3,0)*{};
  (8,2)*{\lambda};(-5,-6)*{\scs i};     (5.1,-6)*{\scs j};
 \endxy
\end{equation}

It is not hard to check using the definition of fake bubbles and the 2-functor $\cal{M}$ that for all values of $\alpha$
\[
 \cal{M}\left( \;\; \vcenter{\xy 0;/r.18pc/:
    (2,-11)*{\icbub{\la i,\lambda\ra-1+\alpha}{i}};
  (12,-2)*{\l};
 \endxy} \;\; \right)
 = c_{i,\l} \;\; \vcenter{\xy 0;/r.18pc/:
    (2,-11)*{\icbub{\la i,\lambda\ra-1+\alpha}{i}};
  (12,-2)*{\l};
 \endxy}
 \qquad \qquad
  \cal{M}\left( \;\; \vcenter{\xy 0;/r.18pc/:
    (2,-11)*{\iccbub{-\la i,\lambda\ra-1+\alpha}{i}};
  (12,-2)*{\l};
 \endxy} \;\; \right)
 = c_{i,\l}^{-1} \;\; \vcenter{\xy 0;/r.18pc/:
    (2,-11)*{\iccbub{-\la i,\lambda\ra-1+\alpha}{i}};
  (12,-2)*{\l};
 \endxy}
\]
One can check that all of the $\mf{sl}_2$ relations will be invariant under this map.

The 2-functor $\cal{M}$ has the obvious inverse using the inverse of the rescalings.
\end{proof}

\begin{cor} \label{cor}
(\cite{KL3,Web5})
If $\Bbbk$ is a field, there is an isomorphism
\begin{eqnarray} \label{def_gamma}
\gamma\maps  \UA & \longrightarrow& K_0(\UcatD_Q^{cyc})
\end{eqnarray}
of linear categories, where $\UcatD_Q^{cyc}$ denotes the Karoubi envelope of the 2-category $\Ucat_Q^{cyc}(\mf{g})$.
\end{cor}

\begin{rem}
Several of the relations in Definition~\ref{defU_cat-cyc} are redundant.  In particular, applying the inverse of the isomorphism $\cal{M}$ from Theorem~\ref{thm} to Brundan's rigidity theorem implies that the cyclicity relations \ref{eq_cyclic_dot-cyc} -- \eqref{eq_crossl-gen-cyc} follow from the others see \cite[Theorem 5.3]{Brundan}.   Furthermore, one can also omit \eqref{eq_biadjoint2-cyc} \cite[Theorem 4.3]{Brundan}.
\end{rem}

%
\section{Helpful relations in $\UcatD_Q^{cyc}$}
\label{sec:relations}
%

%
\subsection{Curl relations}\label{sec:curl}
%

In this section we show that the curl relations follow from those already introduced.  Since this fact is not immediately accessible in the existing literature we provide a proof here.
\begin{lem} \label{lem-partcurl}
The following relations
\[
  \xy 0;/r.17pc/:
  (14,8)*{\l};
  (-3,-10)*{};(3,5)*{} **\crv{(-3,-2) & (2,1)}?(1)*\dir{>};?(.15)*\dir{>};
    (3,-5)*{};(-3,10)*{} **\crv{(2,-1) & (-3,2)}?(.85)*\dir{>} ?(.1)*\dir{>};
  (3,5)*{}="t1";  (9,5)*{}="t2";
  (3,-5)*{}="t1'";  (9,-5)*{}="t2'";
   "t1";"t2" **\crv{(4,8) & (9, 8)};
   "t1'";"t2'" **\crv{(4,-8) & (9, -8)};
   "t2'";"t2" **\crv{(10,0)} ;
   (-6,-8)*{\scs i};
 \endxy
\;\; = \;\;
\left\{
  \begin{array}{ll}
    -c_{i,\l} \Id_{\cal{E}_i\onel}, & \hbox{for $\la i, \l \ra =0$,} \\
    0, & \hbox{for $\la i, \l \ra >0$,}
  \end{array}
\right. \qquad \quad
  \xy 0;/r.17pc/:
  (-14,8)*{\l};
  (3,-10)*{};(-3,5)*{} **\crv{(3,-2) & (-2,1)}?(1)*\dir{>};?(.15)*\dir{>};
    (-3,-5)*{};(3,10)*{} **\crv{(-2,-1) & (3,2)}?(.85)*\dir{>} ?(.1)*\dir{>};
  (-3,5)*{}="t1";  (-9,5)*{}="t2";
  (-3,-5)*{}="t1'";  (-9,-5)*{}="t2'";
   "t1";"t2" **\crv{(-4,8) & (-9, 8)};
   "t1'";"t2'" **\crv{(-4,-8) & (-9, -8)};
   "t2'";"t2" **\crv{(-10,0)} ;
   (6,-8)*{\scs i};
 \endxy \;\; = \;\;
\left\{
  \begin{array}{ll}
    c_{i,\l}^{-1} \Id_{\onel\cal{E}_i}, & \hbox{for $\la i, \l \ra =0$,} \\
    0, & \hbox{for $\la i, \l \ra <0$,}
  \end{array}
\right.
\]
\end{lem}

\begin{proof}
The first claim is proven as follows:
\begin{equation}
  \xy 0;/r.17pc/:
  (14,8)*{\l};
  (-3,-10)*{};(3,5)*{} **\crv{(-3,-2) & (2,1)}?(1)*\dir{>};?(.15)*\dir{>};
    (3,-5)*{};(-3,10)*{} **\crv{(2,-1) & (-3,2)}?(.85)*\dir{>} ?(.1)*\dir{>};
  (3,5)*{}="t1";  (9,5)*{}="t2";
  (3,-5)*{}="t1'";  (9,-5)*{}="t2'";
   "t1";"t2" **\crv{(4,8) & (9, 8)};
   "t1'";"t2'" **\crv{(4,-8) & (9, -8)};
   "t2'";"t2" **\crv{(10,0)} ;
   (-6,-8)*{\scs i};
 \endxy
\;\;  = \;\;
 \vcenter{
  \xy 0;/r.17pc/:
  (12,8)*{\lambda};
   (3,-4)*{};(-3,4)*{} **\crv{(3,-1) & (-3,1)};
   (-3,-4)*{};(3,4)*{} **\crv{(-3,-1) & (3,1)};
   (3,4)*{};(3,12)*{} **\dir{-};
   (-3,4)*{};(-3,12)*{} **\dir{-};
    (-3,12);(-3,16) **\dir{-} ?(1)*\dir{>};
    (-3,-10);(-3,-4) **\dir{-};
  (9,12)*{}="t1";
  (3,12)*{}="t2";
  (9,-4)*{}="t1'";
  (3,-4)*{}="t2'";
   "t1";"t2" **\crv{(9,16) & (3, 16)};
   "t1'";"t2'" **\crv{(9,-9) & (3, -9)};
  "t1'";"t1" **\dir{-} ?(.5)*\dir{<};
  (-6,-8)*{\scs i};
 \endxy}
\;\; \refequal{\eqref{eq_dot_slide_ii-gen-cyc}} \;\;
 \vcenter{
  \xy 0;/r.17pc/:
  (14,8)*{\lambda};
   (3,-4)*{};(-3,4)*{} **\crv{(3,-1) & (-3,1)};
   (-3,-4)*{};(3,4)*{} **\crv{(-3,-1) & (3,1)};
   (3,4)*{};(-3,12)*{} **\crv{(3,7) & (-3,9)};
   (-3,4)*{};(3,12)*{} **\crv{(-3,7) & (3,9)};
    (-3,12);(-3,16) **\dir{-} ?(1)*\dir{>};
    (-3,-10);(-3,-4) **\dir{-};
  (9,12)*{}="t1";
  (3,12)*{}="t2";
  (9,-4)*{}="t1'";
  (3,-4)*{}="t2'";
   "t1";"t2" **\crv{(9,16) & (3, 16)};
   "t1'";"t2'" **\crv{(9,-9) & (3, -9)};
  "t1'";"t1" **\dir{-} ?(.5)*\dir{<};
  (-6,-8)*{\scs i};
  (-3,5)*{\bullet};
 \endxy}
\;\; + \;\;
 \vcenter{
  \xy 0;/r.17pc/:
  (14,8)*{\lambda};
   (3,-4)*{};(-3,4)*{} **\crv{(3,-1) & (-3,1)};
   (-3,-4)*{};(3,4)*{} **\crv{(-3,-1) & (3,1)};
   (3,4)*{};(-3,12)*{} **\crv{(3,7) & (-3,9)};
   (-3,4)*{};(3,12)*{} **\crv{(-3,7) & (3,9)};
    (-3,12);(-3,16) **\dir{-} ?(1)*\dir{>};
    (-3,-10);(-3,-4) **\dir{-};
  (9,12)*{}="t1";
  (3,12)*{}="t2";
  (9,-4)*{}="t1'";
  (3,-4)*{}="t2'";
   "t1";"t2" **\crv{(9,16) & (3, 16)};
   "t1'";"t2'" **\crv{(9,-9) & (3, -9)};
  "t1'";"t1" **\dir{-} ?(.5)*\dir{<};
  (-6,-8)*{\scs i};
  (3,11)*{\bullet};
 \endxy}%
\;\; \refequal{\eqref{eq_r2_ij-gen-cyc}} \;\;
 \vcenter{
  \xy 0;/r.17pc/:
  (14,8)*{\lambda};
   (3,-4)*{};(-3,4)*{} **\crv{(3,-1) & (-3,1)};
   (-3,-4)*{};(3,4)*{} **\crv{(-3,-1) & (3,1)};
   (3,4)*{};(-3,12)*{} **\crv{(3,7) & (-3,9)};
   (-3,4)*{};(3,12)*{} **\crv{(-3,7) & (3,9)};
    (-3,12);(-3,16) **\dir{-} ?(1)*\dir{>};
    (-3,-10);(-3,-4) **\dir{-};
  (9,12)*{}="t1";
  (3,12)*{}="t2";
  (9,-4)*{}="t1'";
  (3,-4)*{}="t2'";
   "t1";"t2" **\crv{(9,16) & (3, 16)};
   "t1'";"t2'" **\crv{(9,-9) & (3, -9)};
  "t1'";"t1" **\dir{-} ?(.5)*\dir{<};
  (-6,-8)*{\scs i};
  (-3,5)*{\bullet};
 \endxy} \nn
\end{equation}
then using cyclicity this is equal to
\begin{equation}
=\;\;
  \vcenter{
  \xy 0;/r.17pc/:
  (10,8)*{\lambda};
   (3,-4)*{};(-3,4)*{} **\crv{(3,-1) & (-3,1)};
   (-3,-4)*{};(3,4)*{} **\crv{(-3,-1) & (3,1)};
   (3,4)*{};(-3,12)*{} **\crv{(3,7) & (-3,9)};
   (-3,4)*{};(3,12)*{} **\crv{(-3,7) & (3,9)};
    (3,12);(3,16) **\dir{-} ?(1)*\dir{>};
    (3,-10);(3,-4) **\dir{-};
  (-9,12)*{}="t1";
  (-3,12)*{}="t2";
  (-9,-4)*{}="t1'";
  (-3,-4)*{}="t2'";
   "t1";"t2" **\crv{(-9,16) & (-3, 16)};
   "t1'";"t2'" **\crv{(-9,-9) & (-3, -9)};
  "t1'";"t1" **\dir{-} ?(.75)*\dir{>};
  (6,-8)*{\scs i};
  (-9,3)*{\bullet};
 \endxy}
\;\; \refequal{\eqref{eq_FE-cyc}} \;\;
-\;
\vcenter{
\xy 0;/r.17pc/:
 (-6,-2)*{\icbub{}{i}};
 (3,-10);(3,10) **\dir{-} ?(1)*\dir{>};
 (6,-8)*{\scs i};
 (10,6)*{\lambda};
 \endxy}
\end{equation}
and the claim follows since a clockwise oriented $i$-labelled bubble with a single dot in weight $\lambda+\alpha_i$ vanishes if $\la i, \l\ra >0$ and is equal to $c_{i,\l+\alpha_i}=c_{i,\l}$.  The second claim is proven similarly.
\end{proof}

\begin{lem} The following relations
 \begin{equation}
  \xy 0;/r.17pc/:
  (14,8)*{\l};
  (-3,-10)*{};(3,5)*{} **\crv{(-3,-2) & (2,1)}?(1)*\dir{>};?(.15)*\dir{>};
    (3,-5)*{};(-3,10)*{} **\crv{(2,-1) & (-3,2)}?(.85)*\dir{>} ?(.1)*\dir{>};
  (3,5)*{}="t1";  (9,5)*{}="t2";
  (3,-5)*{}="t1'";  (9,-5)*{}="t2'";
   "t1";"t2" **\crv{(4,8) & (9, 8)};
   "t1'";"t2'" **\crv{(4,-8) & (9, -8)};
   "t2'";"t2" **\crv{(10,0)} ;
   (-6,-8)*{\scs i};
 \endxy \;\; =
 \;\; -
   \sum_{ \xy  (0,3)*{\scs f_1+f_2+f_3}; (0,0)*{\scs =-\la i,\l\ra};\endxy}
 \xy 0;/r.17pc/:
  (19,4)*{\l};
  (0,0)*{\bbe{}};(-2,-8)*{\scs };
  (-2,-8)*{\scs i};
  (12,-2)*{\icbub{\la i,\l\ra-1+f_2}{i}};
  (0,6)*{\bullet}+(3,-1)*{\scs f_1};
 \endxy
\qquad
  \xy 0;/r.17pc/:
  (-14,8)*{\l};
  (3,-10)*{};(-3,5)*{} **\crv{(3,-2) & (-2,1)}?(1)*\dir{>};?(.15)*\dir{>};
    (-3,-5)*{};(3,10)*{} **\crv{(-2,-1) & (3,2)}?(.85)*\dir{>} ?(.1)*\dir{>};
  (-3,5)*{}="t1";  (-9,5)*{}="t2";
  (-3,-5)*{}="t1'";  (-9,-5)*{}="t2'";
   "t1";"t2" **\crv{(-4,8) & (-9, 8)};
   "t1'";"t2'" **\crv{(-4,-8) & (-9, -8)};
   "t2'";"t2" **\crv{(-10,0)} ;
   (6,-8)*{\scs i};
 \endxy \;\; = \;\;
 \sum_{ \xy  (0,3)*{\scs g_1+g_2+g_3}; (0,0)*{\scs =\la i,\l\ra};\endxy}
  \xy 0;/r.17pc/:
  (5,-8)*{\scs i};
  (-12,8)*{\l};
  (3,0)*{\bbe{}};(2,-8)*{\scs};
  (-12,-2)*{\iccbub{-\la i,\l\ra-1+g_2}{i}};
  (3,6)*{\bullet}+(3,-1)*{\scs g_1};
 \endxy
\end{equation}
hold in $\Ucat_Q^{cyc}(\mf{g})$.
\end{lem}

\begin{proof}
 Note that when either summation is nonzero then the bubbles appearing in the equation are fake bubbles.  We prove the first identity. The second is proven similarly.

When $\la i,\l\ra \geq 0$ this claim follows from Lemma~\ref{lem-partcurl}.  The case $\la i,\l\ra<0$ can be proven by capping of \eqref{eq_FE-cyc} with $-\la i,\l \ra$ dots and simplifying the resulting equation. For more details see for example Section 3.6.1 of \cite{Lau3}.
\end{proof}

Inductively using \eqref{eq_dot_slide_ii-gen-cyc} it is easy to show the following Proposition.
\begin{prop}
\begin{equation}
  \xy 0;/r.17pc/:
  (14,8)*{\l};
  (-3,-10)*{};(3,5)*{} **\crv{(-3,-2) & (2,1)}?(1)*\dir{>};?(.15)*\dir{>};
    (3,-5)*{};(-3,10)*{} **\crv{(2,-1) & (-3,2)}?(.85)*\dir{>} ?(.1)*\dir{>};
  (3,5)*{}="t1";  (9,5)*{}="t2";
  (3,-5)*{}="t1'";  (9,-5)*{}="t2'";
   "t1";"t2" **\crv{(4,8) & (9, 8)};
   "t1'";"t2'" **\crv{(4,-8) & (9, -8)};
   "t2'";"t2" **\crv{(10,0)} ;
(9,0)*{\bullet}+(3,2)*{\scs m};
   (-6,-8)*{\scs i};
 \endxy \;\; =
 \;\; -
   \sum_{ \xy  (0,3)*{\scs f_1+f_2}; (0,0)*{\scs =m-\la i,\l\ra};\endxy}
 \xy 0;/r.17pc/:
  (19,4)*{\l};
  (0,0)*{\bbe{}};(-2,-8)*{\scs };
  (-2,-8)*{\scs i};
  (12,-2)*{\icbub{\la i,\l\ra-1+f_2}{i}};
  (0,6)*{\bullet}+(3,-1)*{\scs f_1};
 \endxy
\qquad
  \xy 0;/r.17pc/:
  (-14,8)*{\l};
  (3,-10)*{};(-3,5)*{} **\crv{(3,-2) & (-2,1)}?(1)*\dir{>};?(.15)*\dir{>};
    (-3,-5)*{};(3,10)*{} **\crv{(-2,-1) & (3,2)}?(.85)*\dir{>} ?(.1)*\dir{>};
  (-3,5)*{}="t1";  (-9,5)*{}="t2";
  (-3,-5)*{}="t1'";  (-9,-5)*{}="t2'";
   "t1";"t2" **\crv{(-4,8) & (-9, 8)};
   "t1'";"t2'" **\crv{(-4,-8) & (-9, -8)};
   "t2'";"t2" **\crv{(-10,0)} ;
(-9,0)*{\bullet}+(-3,2)*{\scs m};
   (6,-8)*{\scs i};
 \endxy \;\; = \;\;
 \sum_{ \xy  (0,3)*{\scs g_1+g_2}; (0,0)*{\scs =m+\la i,\l\ra};\endxy}
  \xy 0;/r.17pc/:
  (5,-8)*{\scs i};
  (-12,8)*{\l};
  (3,0)*{\bbe{}};(2,-8)*{\scs};
  (-12,-2)*{\iccbub{-\la i,\l\ra-1+g_2}{i}};
  (3,6)*{\bullet}+(3,-1)*{\scs g_1};
 \endxy
\end{equation}
hold in $\Ucat_Q^{cyc}(\mf{g})$.
\end{prop}

%
\subsection{Bubble slide equations}\label{sec:bubble}
%

In what follows it is often convenient to introduce a shorthand notation
\[
    \xy 0;/r.18pc/:
  (4,8)*{\lambda};
  (2,-2)*{\icbub{\spadesuit+m}{i}};
 \endxy \;\; := \;\;
   \xy 0;/r.18pc/:
  (4,8)*{\lambda};
  (2,-2)*{\icbub{\la i,\lambda\ra-1+m}{i}};
 \endxy
 \qquad
 \qquad
    \xy 0;/r.18pc/:
  (4,8)*{\lambda};
  (2,-2)*{\iccbub{\spadesuit+m}{i}};
 \endxy \;\; := \;\;
   \xy 0;/r.18pc/:
  (4,8)*{\lambda};
  (2,-2)*{\iccbub{-\la i,\lambda\ra-1+m}{i}};
 \endxy
\]
for all $\la i,\lambda\ra$.  Note that as long as $m \geq 0$ this notation makes sense even when $\spadesuit+m <0$.

Using cyclicity, \eqref{eq_r2_ij-gen-cyc} and \eqref{mixed_rel-cyc} one can prove that the following bubble slide equations
\begin{eqnarray*}
    \xy 0;/r.18pc/:
  (14,8)*{\lambda};
  (0,0)*{\bbe{}};
  (0,-12)*{\scs j};
  (12,-2)*{\iccbub{\spadesuit+m}{i}};
  (0,6)*{ }+(7,-1)*{\scs  };
 \endxy
 & \; = \; &
 \left\{
 \begin{array}{ccl}
  \xsum{f=0}{m}(m+1-f)
   \xy 0;/r.18pc/:
  (0,8)*{\lambda+\alpha_j};
  (12,0)*{\bbe{}};
  (12,-12)*{\scs j};
  (0,-2)*{\iccbub{\spadesuit+f}{i}};
  (12,6)*{\bullet}+(5,-1)*{\scs m-f};
 \endxy
    &  & \text{if $i=j$} \\ \\
     t_{ij}\;   \xy 0;/r.18pc/:
  (0,8)*{\lambda+\alpha_j};
  (12,0)*{\bbe{}};
  (11,-12)*{\scs j};
  (0,-2)*{\iccbub{\spadesuit+m}{i}};
 \endxy
 \;\; + \;\;  t_{ji} \;\;
  \xy 0;/r.18pc/:
  (0,8)*{\lambda+\alpha_j};
  (12,0)*{\bbe{}};
  (12,-12)*{\scs j};
  (0,-2)*{\iccbub{\spadesuit+m-d_{ij}}{i}};
  (12,6)*{\bullet}+(5,-1)*{\scs d_{ji}};
 \endxy
 \;\; + \;\; \xsum{p,q}{}s_{i,j}^{p,q} \;
   \xy 0;/r.18pc/:
  (0,8)*{\lambda+\alpha_j};
  (12,0)*{\bbe{}};
  (12,-12)*{\scs j};
  (0,-2)*{\iccbub{\spadesuit+m-d_{ij}+p}{i}};
  (12,6)*{\bullet}+(3,-1)*{\scs q};
 \endxy
   &  & \text{if $a_{ij} <0$} \\
 \qquad \qquad t_{ij} \; \xy 0;/r.18pc/:
  (0,8)*{\lambda+\alpha_j};
  (12,0)*{\bbe{}};
  (12,-12)*{\scs j};
  (0,-2)*{\iccbub{\spadesuit+m}{i}};
 \endxy  &  & \text{if $a_{ij}=0$}
 \end{array}
 \right.
\end{eqnarray*}
\begin{eqnarray*} 
    \xy 0;/r.18pc/:
  (15,8)*{\lambda};
  (11,0)*{\bbe{}};
  (11,-12)*{\scs j};
  (0,-2)*{\icbub{\spadesuit+m\quad }{i}};
 \endxy
   &\; = \; &
  \left\{\begin{array}{ccl}
     \xsum{f=0}{m}(m+1-f)
     \xy 0;/r.18pc/:
  (18,8)*{\lambda};
  (0,0)*{\bbe{}};
  (0,-12)*{\scs j};
  (14,-4)*{\icbub{\spadesuit+f}{i}};
  (0,6)*{\bullet }+(5,-1)*{\scs m-f};
 \endxy
         &  & \text{if $i=j$}  \\
  t_{ji}\;\;  \xy 0;/r.18pc/:
  (18,8)*{\lambda};
  (0,0)*{\bbe{}};
  (0,-12)*{\scs j};
  (12,-2)*{\icbub{\spadesuit+ m-d_{ij}}{i}};
    (0,6)*{\bullet }+(-5,-1)*{\scs d_{ji}};
 \endxy
 \;\;+\;\; t_{ij} \;
\xy 0;/r.18pc/:
  (18,8)*{\lambda};
  (0,0)*{\bbe{}};
  (0,-12)*{\scs j};
  (12,-2)*{\icbub{\spadesuit+ m}{i}};
 \endxy
  \;\;+\;\;
  \xsum{p,q}{} s_{j,i}^{p,q}
  \xy 0;/r.18pc/:
  (18,8)*{\lambda};
  (0,0)*{\bbe{}};
  (0,-12)*{\scs j};
  (0,6)*{\bullet }+(-3,-1)*{\scs p};
  (12,-2)*{\icbub{\spadesuit+ m-d_{ij}+q}{i}};
 \endxy
         & &  \text{if $a_{ij} <0$}\\
t_{ji} \;
    \xy 0;/r.18pc/:
  (15,8)*{\lambda};
  (0,0)*{\bbe{}};
  (0,-12)*{\scs j};
  (12,-2)*{\icbub{\spadesuit+m}{i}};
 \endxy          &  & \text{if $a_{ij} = 0$}
         \end{array}
 \right.
\end{eqnarray*}
hold in $\Ucat_Q^{cyc}(\mf{g})$. Inverting these formulas gives the expressions below.
 \begin{equation*}
      \xy 0;/r.18pc/:
  (15,8)*{\lambda};
  (0,0)*{\bbe{}};
  (0,-12)*{\scs j};
  (12,-2)*{\icbub{\spadesuit+m}{i}};
 \endxy
   =
  \left\{
  \begin{array}{cl}
     \xy 0;/r.18pc/:
  (0,8)*{\lambda+\alpha_j};
  (12,0)*{\bbe{}};
  (12,-12)*{\scs j};
  (0,-2)*{\icbub{\spadesuit+(m-2)}{i}};
  (12,6)*{\bullet}+(3,-1)*{\scs 2};
 \endxy
   -2 \;
         \xy 0;/r.18pc/:
  (0,8)*{\lambda+\alpha_j};
  (12,0)*{\bbe{}};
  (12,-12)*{\scs j};
  (0,-2)*{\icbub{\spadesuit+(m-1)}{i}};
  (12,6)*{\bullet}+(8,-1)*{\scs };
 \endxy
 + \;\;
     \xy 0;/r.18pc/:
  (0,8)*{\lambda+\alpha_j};
  (12,0)*{\bbe{}};
  (12,-12)*{\scs j};
  (0,-2)*{\icbub{\spadesuit+m}{i}};
  (12,6)*{}+(8,-1)*{\scs };
 \endxy
  &   \text{if $i = j$} \\
  -\;\xsum{ \xy (0,4)*{\;};(0,-1)*{\scs f\geq 0}; (0,-4.3)*{\scs k \geq 0}; \endxy}{} \binom{f+k}{k} (-t_{ij}^{-1}t_{ji})^{f}(-t_{ij})^{-(k+1)}
\xsum{ \xy (0,4)*{\;};(0,-1)*{\scs p_1,\ldots,p_k\geq 0}; (0,-4.3)*{\scs q_1,\ldots,q_k\geq 0}; \endxy}{}
  \left(\xprod{\ell=1}{k}s_{j,i}^{p_{\ell},q_{\ell}} \right)
  \xy 0;/r.18pc/:
  (18,2)*{\lambda};
  (14,0)*{\bbe{}};
  (14,-12)*{\scs j};
  (-6,-2)*{\icbub{\spadesuit+m-fd_{ij} +q_1 + \dots +q_{k}}{i}};
  (14,6)*{\bullet}+(-14,2)*{\scs fd_{ji} + p_1 + \dots + p_{k} };
 \endxy &   \text{if $a_{ij} <0$}
  \end{array}
 \right.
\end{equation*}
\begin{equation*} \label{cc_slide_right}
    \xy 0;/r.18pc/:
  (0,8)*{\lambda+\alpha_j};
  (12,0)*{\bbe{}};
  (12,-12)*{\scs j};
  (0,-2)*{\iccbub{\spadesuit+m}{i}};
  (12,6)*{}+(8,-1)*{\scs };
 \endxy \hspace{-0.1in}
  =
\left\{
\begin{array}{cc}
    \xy 0;/r.18pc/:
  (15,8)*{\lambda};
  (0,0)*{\bbe{}};
  (0,-12)*{\scs j};
  (12,-2)*{\iccbub{\spadesuit+(m-2)}{i}};
  (0,6)*{\bullet }+(3,1)*{\scs 2};
 \endxy
  -2 \;
      \xy 0;/r.18pc/:
  (15,8)*{\lambda};
  (0,0)*{\bbe{}};
  (0,-12)*{\scs j};
  (12,-2)*{\iccbub{\spadesuit+(m-1)}{i}};
  (0,6)*{\bullet }+(5,-1)*{\scs };
 \endxy
 + \;\;
      \xy 0;/r.18pc/:
  (15,8)*{\lambda};
  (0,0)*{\bbe{}};
  (0,-12)*{\scs j};
  (12,-2)*{\iccbub{\spadesuit+m}{i}};
 \endxy
  &   \text{if $i=j$} \\
-\;\xsum{ \xy (0,4)*{\;};(0,-1)*{\scs f\geq 0}; (0,-4.3)*{\scs k \geq 0}; \endxy}{} \binom{f+k}{k} (-t_{ij}^{-1}t_{ji})^{f}(-t_{ij})^{-(k+1)}
\xsum{ \xy (0,4)*{\;};(0,-1)*{\scs p_1,\ldots,p_k\geq 0}; (0,-4.3)*{\scs q_1,\ldots,q_k\geq 0}; \endxy}{}
  \left(\xprod{\ell=1}{k}s_{i,j}^{q_{\ell},p_{\ell}} \right)
    \xy 0;/r.18pc/:
  (29,0)*{\lambda};
  (0,0)*{\bbe{}};
  (0,-12)*{\scs j};
  (23,-2)*{\iccbub{\spadesuit+m-fd_{ij}+q_1+\dots +q_k}{i}};
  (0,6)*{\bullet }+(15,1)*{\scs fd_{ji}+p_1+\dots+p_k};
 \endxy
    &   \text{if $a_{ij} <0$}
\end{array}
\right.
\end{equation*}

%
\subsection{Triple intersections}\label{sec:triple}
%

Using cyclicity and \eqref{eq_r3_easy-gen-cyc} one can show that for
all $i,j,k \in I$ not satisfying $i=j=k$ the equation
\begin{equation}
\vcenter{\xy 0;/r.17pc/:
    (-4,-4)*{};(4,4)*{} **\crv{(-4,-1) & (4,1)}?(0)*\dir{};
    (4,-4)*{};(-4,4)*{} **\crv{(4,-1) & (-4,1)}?(0)*\dir{<};
    (4,4)*{};(12,12)*{} **\crv{(4,7) & (12,9)}?(1)*\dir{};
    (12,4)*{};(4,12)*{} **\crv{(12,7) & (4,9)}?(1)*\dir{};
    (-4,12)*{};(4,20)*{} **\crv{(-4,15) & (4,17)}?(1)*\dir{};
    (4,12)*{};(-4,20)*{} **\crv{(4,15) & (-4,17)}?(1)*\dir{};
    (-4,4)*{}; (-4,12) **\dir{-};
    (12,-4)*{}; (12,4) **\dir{-};
    (12,12)*{}; (12,20) **\dir{-};
    (4,20); (4,21) **\dir{-}?(1)*\dir{};
    (-4,20); (-4,21) **\dir{-}?(1)*\dir{>};
    (12,20); (12,21) **\dir{-}?(1)*\dir{>};
   (18,8)*{\lambda};  (-6,-3)*{\scs i};
  (6.5,-3)*{\scs j};
  (15,-3)*{\scs k};
\endxy}
 \;\; =\;\;
\vcenter{\xy 0;/r.17pc/:
    (4,-4)*{};(-4,4)*{} **\crv{(4,-1) & (-4,1)}?(0)*\dir{};
    (-4,-4)*{};(4,4)*{} **\crv{(-4,-1) & (4,1)}?(0)*\dir{<};
    (-4,4)*{};(-12,12)*{} **\crv{(-4,7) & (-12,9)}?(1)*\dir{};
    (-12,4)*{};(-4,12)*{} **\crv{(-12,7) & (-4,9)}?(1)*\dir{};
    (4,12)*{};(-4,20)*{} **\crv{(4,15) & (-4,17)}?(1)*\dir{};
    (-4,12)*{};(4,20)*{} **\crv{(-4,15) & (4,17)}?(1)*\dir{};
    (4,4)*{}; (4,12) **\dir{-};
    (-12,-4)*{}; (-12,4) **\dir{-};
    (-12,12)*{}; (-12,20) **\dir{-};
    (4,20); (4,21) **\dir{-}?(1)*\dir{>};
    (-4,20); (-4,21) **\dir{-}?(1)*\dir{};
    (-12,20); (-12,21) **\dir{-}?(1)*\dir{>};
  (10,8)*{\lambda};
  (-14,-3)*{\scs i};
  (-6.5,-3)*{\scs j};
  (6,-3)*{\scs k};
\endxy}
\end{equation}
holds in $\Ucat_Q^{cyc}(\mf{g})$.  Similarly, it is not hard to show that if $(\alpha_i,\alpha_j)<0$, then
\begin{equation}
\;\; \vcenter{
 \xy 0;/r.17pc/:
    (-4,-4)*{};(4,4)*{} **\crv{(-4,-1) & (4,1)}?(0)*\dir{<};
    (4,-4)*{};(-4,4)*{} **\crv{(4,-1) & (-4,1)}?(0)*\dir{};
    (4,4)*{};(12,12)*{} **\crv{(4,7) & (12,9)}?(1)*\dir{};
    (12,4)*{};(4,12)*{} **\crv{(12,7) & (4,9)}?(1)*\dir{};
    (-4,12)*{};(4,20)*{} **\crv{(-4,15) & (4,17)}?(1)*\dir{>};
    (4,12)*{};(-4,20)*{} **\crv{(4,15) & (-4,17)}?(1)*\dir{>};
    (-4,4)*{}; (-4,12) **\dir{-};
    (12,-4)*{}; (12,4) **\dir{-};
    (12,12)*{}; (12,20) **\dir{-}?(1)*\dir{};
  (-6,-3)*{\scs i};
  (6,-3)*{\scs i};
  (14,-3)*{\scs j};
\endxy}
\quad - \quad
 \vcenter{
 \xy 0;/r.17pc/:
    (4,-4)*{};(-4,4)*{} **\crv{(4,-1) & (-4,1)}?(1)*\dir{};
    (-4,-4)*{};(4,4)*{} **\crv{(-4,-1) & (4,1)}?(1)*\dir{};
    (-4,4)*{};(-12,12)*{} **\crv{(-4,7) & (-12,9)}?(1)*\dir{};
    (-12,4)*{};(-4,12)*{} **\crv{(-12,7) & (-4,9)}?(1)*\dir{};
    (4,12)*{};(-4,20)*{} **\crv{(4,15) & (-4,17)}?(1)*\dir{>};
    (-4,12)*{};(4,20)*{} **\crv{(-4,15) & (4,17)}?(1)*\dir{};
    (4,4)*{}; (4,12) **\dir{-};
    (-12,-4)*{}; (-12,4) **\dir{-}?(0)*\dir{<};
    (-12,12)*{}; (-12,20) **\dir{-}?(1)*\dir{>};
  (6,-3)*{\scs j};
  (-6,-3)*{\scs i};
  (-14,-3)*{\scs i};
\endxy}\;\;
 \;\; =\;\;
 -t_{ij}\sum_{\ell_1+\ell_2=d_{ij}-1} \;\;
 \vcenter{
\xy 0;/r.17pc/:
  (4,-4)*{};(-4,-4)*{} **\crv{(4,1) & (-4,1)}?(1)*\dir{>} ?(.5)*\dir{}+(0,0)*{\bullet}+(2,3)*{\scs \ell_1};
  (12,20)*{};(4,20)*{} **\crv{(12,15) & (4,15)}?(1)*\dir{>} ?(.5)*\dir{}+(0,0)*{\bullet}+(1,-3)*{\scs \ell_2};
  (12,-4)*{};(-4,20)*{} **\crv{(12,6) & (-4,14)}?(1)*\dir{>} ;
  (-6,-4)*{\scs i};     (6.1,-4)*{\scs i};
  (14,-4)*{\scs j};
 \endxy}
 \;\; - \;\;
 \sum_{p,q} s_{ij}^{pq} \sum_{\xy (0,2.3)*{\scs \ell_1+\ell_2}; (0,-1)*{\scs =p-1}; \endxy} \;
 \vcenter{
\xy 0;/r.17pc/:
  (4,-4)*{};(-4,-4)*{} **\crv{(4,1) & (-4,1)}?(1)*\dir{>} ?(.5)*\dir{}+(0,0)*{\bullet}+(2,3)*{\scs \ell_1};
  (12,20)*{};(4,20)*{} **\crv{(12,15) & (4,15)}?(1)*\dir{>} ?(.5)*\dir{}+(0,0)*{\bullet}+(1,-3)*{\scs \ell_2};
  (12,-4)*{};(-4,20)*{} **\crv{(12,6) & (-4,14)}?(1)*\dir{>}
    ?(.35)*\dir{}+(0,0)*{\bullet}+(2,3)*{\scs q};
  (-6,-4)*{\scs i};     (6.1,-4)*{\scs i};
  (14,-4)*{\scs j};
 \endxy}
\end{equation}
\begin{equation}
\;\; \vcenter{
 \xy 0;/r.17pc/:
    (-4,-4)*{};(4,4)*{} **\crv{(-4,-1) & (4,1)}?(0)*\dir{};
    (4,-4)*{};(-4,4)*{} **\crv{(4,-1) & (-4,1)}?(0)*\dir{};
    (4,4)*{};(12,12)*{} **\crv{(4,7) & (12,9)}?(0)*\dir{};
    (12,4)*{};(4,12)*{} **\crv{(12,7) & (4,9)}?(1)*\dir{};
    (-4,12)*{};(4,20)*{} **\crv{(-4,15) & (4,17)}?(1)*\dir{>};
    (4,12)*{};(-4,20)*{} **\crv{(4,15) & (-4,17)}?(1)*\dir{};
    (-4,4)*{}; (-4,12) **\dir{-};
    (12,-4)*{}; (12,4) **\dir{-}?(0)*\dir{<};
    (12,12)*{}; (12,20) **\dir{-}?(1)*\dir{>};
  (-6,-3)*{\scs j};
  (6,-3)*{\scs i};
  (14,-3)*{\scs i};
\endxy}
\quad - \quad
 \vcenter{
 \xy 0;/r.17pc/:
    (4,-4)*{};(-4,4)*{} **\crv{(4,-1) & (-4,1)}?(0)*\dir{<};
    (-4,-4)*{};(4,4)*{} **\crv{(-4,-1) & (4,1)}?(1)*\dir{};
    (-4,4)*{};(-12,12)*{} **\crv{(-4,7) & (-12,9)}?(1)*\dir{};
    (-12,4)*{};(-4,12)*{} **\crv{(-12,7) & (-4,9)}?(1)*\dir{};
    (4,12)*{};(-4,20)*{} **\crv{(4,15) & (-4,17)}?(1)*\dir{>};
    (-4,12)*{};(4,20)*{} **\crv{(-4,15) & (4,17)}?(1)*\dir{>};
    (4,4)*{}; (4,12) **\dir{-};
    (-12,-4)*{}; (-12,4) **\dir{-}?(0)*\dir{};
    (-12,12)*{}; (-12,20) **\dir{-}?(1)*\dir{};
  (6,-3)*{\scs i};
  (-6,-3)*{\scs i};
  (-14,-3)*{\scs j};
\endxy}\;\;
 \;\; =\;\;
 -t_{ij}\sum_{\ell_1+\ell_2=d_{ij}-1} \;\;
 \vcenter{
\xy 0;/r.17pc/:
  (-4,-4)*{};(4,-4)*{} **\crv{(-4,1) & (4,1)}?(1)*\dir{>} ?(.5)*\dir{}+(0,0)*{\bullet}+(-2,3)*{\scs \ell_1};
  (-12,20)*{};(-4,20)*{} **\crv{(-12,15) & (-4,15)}?(1)*\dir{>} ?(.5)*\dir{}+(0,0)*{\bullet}+(-1,-3)*{\scs \ell_2};
  (-12,-4)*{};(4,20)*{} **\crv{(-12,6) & (4,14)}?(1)*\dir{>} ;
  (6,-4)*{\scs j};     (-6.1,-4)*{\scs i};
  (-14,-4)*{\scs i};
 \endxy}
 \;\; - \;\;
 \sum_{p,q} s_{ij}^{pq} \sum_{\xy (0,2.3)*{\scs \ell_1+\ell_2};   (0,-1)*{\scs =p-1}; \endxy} \;
 \vcenter{
\xy 0;/r.17pc/:
    (-4,-4)*{};(4,-4)*{} **\crv{(-4,1) & (4,1)}?(1)*\dir{>} ?(.5)*\dir{}+(0,0)*{\bullet}+(-2,3)*{\scs \ell_1};
  (-12,20)*{};(-4,20)*{} **\crv{(-12,15) & (-4,15)}?(1)*\dir{>} ?(.5)*\dir{}+(0,0)*{\bullet}+(-1,-3)*{\scs \ell_2};
  (-12,-4)*{};(4,20)*{} **\crv{(-12,6) & (4,14)}?(1)*\dir{>}
    ?(.35)*\dir{}+(0,0)*{\bullet}+(-2,3)*{\scs q};
  (6,-4)*{\scs j};     (-6.1,-4)*{\scs i};
  (-14,-4)*{\scs i};
 \endxy}
\end{equation}
also hold in $\Ucat_Q^{cyc}(\mf{g})$.



%

\end{document}